\documentclass[12pt, reqno]{amsart} 
\usepackage{amssymb,amscd,amsfonts,amsbsy}
\usepackage{latexsym}
\usepackage{exscale}
\usepackage{amsmath,amsthm,amsfonts}
\usepackage{mathrsfs}
\usepackage{xcolor} 
\usepackage[colorlinks=true,linkcolor=blue,citecolor=red,urlcolor=red, backref=page]{hyperref} 
\usepackage{esint} 
\usepackage{stmaryrd}
\usepackage{pifont}

\usepackage[utf8]{inputenc}

\calclayout

\allowdisplaybreaks

\newtheorem{theorem}{Theorem}[section]

\newtheorem{lemma}[theorem]{Lemma}

\newtheorem{definition}[theorem]{Definition}

\numberwithin{equation}{section} 

\def\A{\mathcal{A}}

\def\D{\mathcal{D}}
\def\M{\mathcal{M}}

\def\R{\mathbb{R}}
\def\Rn{\mathbb{R}^n}
\def\S{\mathcal{S}}
\def\Z{\mathbb{Z}}

\def\supp{\operatorname{supp}}

\begin{document}

\title[The multilinear Littlewood-Paley operators]
  {Weak and strong type estimates for the multilinear Littlewood-Paley operators}

\author[Cao]{Mingming Cao}
\address{Mingming Cao\\
Instituto de Ciencias Matem\'aticas CSIC-UAM-UC3M-UCM\\
Con\-se\-jo Superior de Investigaciones Cient{\'\i}ficas\\
C/ Nicol\'as Cabrera, 13-15\\
E-28049 Ma\-drid, Spain} \email{mingming.cao@icmat.es}

\author[Hormozi]{Mahdi Hormozi}
\address{Mahdi Hormozi\\
School of Mathematics\\
Institute for Research in Fundamental Sciences (IPM)\\
P. O. Box 19395-5746, Tehran, Iran} \email{me.hormozi@gmail.com}

\author[Iba\~nez-Firnkorn]{Gonzalo Iba\~nez-Firnkorn}
\address{Gonzalo Iba\~nez-Firnkorn\\
FAMAF-CIEM (CONICET),
Universidad Nacional de C\'ordoba,
Medina A\-llen\-de s/n, Ciudad Universitaria\\
C\' ordoba, Argentina.}
\email{gibanez@famaf.unc.edu.ar}

\author[Rivera-R\'ios]{Israel P. Rivera-R\'ios}
\address{Israel P. Rivera-R\'ios\\
Departamento de Matemática, Universidad Nacional del Sur (UNS), Bahía
Blanca, Argentina and INMABB, Universidad Nacional del Sur (UNS)-CONICET,
Alem 1253\\
Bahía Blanca, Argentina.}\email{israelpriverarios@gmail.com}

\author[Si]{Zengyan Si}
\address{Zengyan Si\\
School of Mathematics and Information Science\\
Henan Polytechnic University\\
Jiaozuo 454000\\
People's Republic of China} \email{zengyan@hpu.edu.cn}

\author[Yabuta]{K\^{o}z\^{o} Yabuta}
\address{K\^{o}z\^{o} Yabuta\\
Research Center for Mathematics and Data Science\\
Kwansei Gakuin University\\
Gakuen 2-1, Sanda 669-1337\\
Japan}\email{kyabuta3@kwansei.ac.jp}

\thanks{M. C. acknowledges financial support from the Spanish Ministry of Science and Innovation, through the ``Severo Ochoa Programme for Centres of Excellence in R\&D'' (SEV-2015-0554) and from the Spanish National Research Council, through the ``Ayuda extraordinaria a Centros de Excelencia Severo Ochoa'' (20205CEX001). M. H. is supported by a grant from IPM.  G. I.-F. is partially supported by CONICET and SECYT-UNC. I. P. R.-R. is partially supported by CONICET PIP 11220130100329CO and Agencia I+D+i PICT 2018-02501. Z. S. is supported partly by the Key Research Project for Higher Education in Henan Province(No.19A110017) and the Fundamental Research Funds for the Universities of Henan Province(No.NSFRF200329).}


\date{October 22, 2020}

\begin{abstract}
Let $S_{\alpha}$ be the multilinear square function defined on the cone with aperture $\alpha \geq 1$. In this paper, we investigate several kinds of weighted norm inequalities for $S_{\alpha}$. We first obtain a sharp weighted estimate in terms of aperture $\alpha$ and $\vec{w} \in A_{\vec{p}}$. By means of some pointwise estimates, we also establish two-weight inequalities including bump and entropy bump estimates, and Fefferman-Stein inequalities with arbitrary weights. Beyond that, we consider the mixed weak type estimates corresponding Sawyer's conjecture, for which a Coifman-Fefferman inequality with the precise $A_{\infty}$ norm is proved. Finally, we present the local decay estimates using the extrapolation techniques and dyadic analysis respectively. All the conclusions aforementioned hold for the Littlewood-Paley $g^*_{\lambda}$ function. Some results are new even in the linear case.
\end{abstract}

\maketitle

\section{Introduction}\label{sec:intro}

Given $\alpha>0$, let $S_{\alpha}$ be the square function defined by
$$
S_{\alpha}(f)(x)=\bigg(\iint_{\Gamma_\alpha(x)}|f* \psi_t(f)(y)|^2\frac{dydt}{t^{n+1}}\bigg)^{\frac12},
$$
 where $\psi_t(x)=t^{-n} \psi(x/t)$ and $\Gamma_{\alpha}(x)$ is the cone at vertex $x$ with aperture $\alpha$. Lerner \cite{Ler11}, by applying the intrinsic square function introduced in \cite{Wil}, proved sharp weighted norm inequalities for $S_{\alpha}(f)$. Later on, he improved the result in the sense of determination of sharp dependence on $\alpha$ in \cite{Ler14} by using the local mean oscillation formula. More precisely,
 \begin{equation}
\label{eq:sharp_lsq}
  \|S_{\alpha}\|_{L^p(w)\to L^p(w)}
\lesssim \alpha^n [w]_{A_p}^{\max{\{\frac{1}{2},\frac{1}{p-1}\}}},
\,\,  1<p<\infty.
\end{equation}
The preceding result is among the plenty important results in the fruitful realm of weighted inequalities concerning the precise determination of the optimal bounds of the weighted operator norm of different singular integral operators. We refer the interested reader to \cite{Hy1, Hy2, Lacey, Ler13} and the references therein for a survey on the advances on the topic.

 Let us recall the definition of multilinear square functions considered in this paper. The standard kernel for multilinear square functions was introduced in \cite{SXY}. Let $\psi(x,\vec{y}):=\psi(x, y_1, \ldots, y_m)$ be a locally integrable function defined away from the diagonal $x =y_1=\cdots =y_m$ in $(\mathbb{R}^n)^{m+1}$. We assume that there are positive constants $\delta$ and $A$ so that the following conditions hold:

\begin{itemize}
\item Size condition: $$|\psi(x, \vec{y})| \leq \frac{A}{\left( 1+\sum_{i=1}^m |x-y_i|\right)^{mn+\delta}}.$$
\item Smoothness condition: There exists $\gamma>0$ so that
\[
|\psi(x, \vec{y})-\psi(x', \vec{y})|
\leq \frac{A|x-x'|^\gamma}{\left(1+\sum_{i=1}^m |x-y_i|\right)^{mn+\delta+\gamma}},
\]
whenever $|x-x'|< \frac{1}{2}\max_j|x-y_j|$, and
\[
\left|\psi(x, \vec{y})-\psi(x,y_1,\ldots,y'_i,\ldots,y_m)\right|
 \leq \frac{A|y_i-y'_i|^\gamma}{\left( 1+\sum_{i=1}^m |x-y_i|\right)^{mn+\delta+\gamma}},
\]
whenever $|y_i-y'_i|< \frac{1}{2}\max_j|x-y_j|$ for $i=1,2,\dots,m$.
\end{itemize}
For $t>0$, denote $\psi_t$
\[
\psi_t(\vec{f})(x) :=\frac{1}{t^{mn}}\int_{(\Rn)^m}\psi
\Big(\frac{x}{t},\frac{y_1}{t},\cdots, \frac{y_m}{t}\Big)\prod_{j=1}^m f_j(y_j)dy_j,
\]
for all $x\notin \bigcap_{j=1}^m {\rm supp}\, f_j$ and $\vec{f}=(f_1,\ldots,f_m)\in \S(\Rn) \times \cdots \times \S(\Rn)$.

Given $ \alpha>0$ and $\lambda>2m$, the multilinear square functions $S_{\alpha}$ and $g^*_{\lambda}$ are defined by
\begin{align*}
S_{\alpha}(\vec{f})(x) :=\bigg(\iint_{\Gamma_\alpha(x)}|\psi_t(\vec{f})(y)|^2\frac{dydt}{t^{n+1}}\bigg)^{1/2},
\end{align*}
where $\Gamma_\alpha(x)=\{(y,t)\in \mathbb{R}^{n+1}_+: |x-y|<\alpha t\}$, and
\begin{align*}
g^*_{\lambda}(\vec{f})(x) :=\bigg(\iint_{\mathbb{R}^{n+1}_+}\Big(\frac{t}{t+|x-y|}\Big)^{n\lambda}|\psi_t(\vec{f})(y)|^2\frac{dydt}{t^{n+1}}\bigg)^{1/2}.
\end{align*}
Hereafter, we assume that for $\lambda>2m$ there exist some $1\le p_1,\dots, p_m\le \infty$ and some $0<p<\infty$ with $\frac 1p=\frac{1}{p_1}+\cdots+\frac{1}{p_m}$, such that $g^*_{\lambda}$ maps continuously $L^{p_1}(\Rn)\times\cdots\times L^{p_m}(\Rn)$ to $L^p(\Rn)$. Under this condition, it was proved in \cite{SXY} that {\color{red}$g^*_{\lambda}$} maps continuously $L^1(\Rn) \times \cdots \times L^1(\Rn) \rightarrow L^{1/m,\infty}(\Rn)$ provided $\lambda > 2m$. Moreover,
since $S_{\alpha}$ is dominated by $g_{\lambda}^{*}$, we also get that {\color{red}$S_\alpha$} maps continuously $L^1(\Rn) \times \cdots \times L^1(\Rn) \rightarrow L^{1/m,\infty}(\Rn)$.

These two mutilinear square functions were introduced and investigated in \cite{SXY, XY}. Indeed, the theory of multilinear Littlewood-Paley operators originated in the works of Coifman and Meyer \cite{CM}. The multilinear square functions has important applications in PDEs and other fields. In particular, Fabes, Jerison, and Kenig brought very important applications of multilinear square functions in PDEs to the attention. {\color{red}In \cite{FJK1}}, they studied the solutions of Cauchy problem for non-divergence form parabolic equations by obtaining  some multilinear Littlewood-Paley type estimates for the square root of an elliptic operator in divergence form. Also, the necessary and sufficient conditions for absolute continuity of elliptic-harmonic measure were achieved relying upon a multilinear Littlewood-Paley estimate, in \cite{FJK2}. Moreover, in  \cite{FJK3}, they applied a class of multilinear square functions to Kato's problem. For further details on the theory of multilinear square functions and their applications, we refer to \cite{CY, CDM, CMM, CM, FJK1, FJK3} and the references therein.

In this paper, we investigate some weak and strong type estimates for multilinear Littlewood-Paley operators. This kind of inequalities has its origin in classical potential theory. A big breakthrough in understanding Poisson's equation, made by Lichtenstein \cite{Lich} in 1916, raised problems that have been central to analysis over the past decades. The theory of singular integral operators owes its impetus to the change of point of view of potential theory generated by this work. The action of singular integral operators on the standard Lebesgue spaces $L^p(\Rn)$ was for a long time the main object of study. But these operators have natural analogs in which $\Rn$ is replaced by a Lie group or Lebesgue measure on $\Rn$ is replaced by a weighted measure. It is in the setting that our work is focused on.

The contributions of this paper are as follows. Based on the ideas from Fefferman's celebrated paper \cite{F}, in this work, we first prove the upper bound for $S_\alpha$ is sharp  in the aperture $\alpha$ on all class $A_{\vec{p}}$. Secondly, we focus on bump and entropy bump estimates, mixed weak type estimates, local decay estimates,  and multilinear version of Fefferman-Stein inequality with arbitrary weights for multilinear square functions respectively. These interesting estimates have aroused the attention of many researchers. For example, $A_p$ bump conditions may be thought of as the classical two-weight $A_p$ condition with the localized $L^p$ and $L^{p'}$ norms "bumped up" in the scale of Orlicz spaces. These conditions have a long history, we refer to \cite{F2,PW}. Muckenhoupt and Wheeden \cite{MW} first formulated the mixed weak type estimates for Hardy-Littlewood maximal function and the Hilbert transform on the real line although Sawyer \cite{S83} considered a more singular case, namely he showed that if $\mu\in A_1$ and $\nu\in A_\infty$, then

\begin{equation}\label{eq:Sawyer}
\bigg\|\frac{M(f\nu)}{\nu}\bigg\|_{L^{1,\infty}(\mu\nu)} \lesssim\|f\|_{L^1(uv)}
\end{equation}
and conjectured that such an inequality should hold with $M$ replaced by the Hilbert transform. Later on Cruz-Uribe, Martell and P\'{e}rez \cite{CMP} extended Sawyer's result to higher dimensions and also settled Sawyer's conjecture and extended that result for general Calder\'on-Zygmund operators reducing it to the case of maximal functions via an extrapolation argument. That extrapolation argument allowed them to take $\mu\in A_1$ and $\nu\in A_\infty$. That led them to conjecture that \eqref{eq:Sawyer} should hold  $\mu\in A_1$ and $\nu\in A_\infty$. Recently, that conjecture was settled by Li, Ombrosi and P\'{e}rez \cite{LOP}. That result was extended to maximal operators with Young functions \cite{B2}. Analogous results were obtained for commutators \cite{BCP}, fractional operators \cite{BCP2} or in the multilinear setting \cite{LOPi}. Also quantitative estimates have been studied in \cite{CRR} and \cite{OmPR}. Local exponential decay estimates for CZOs and square functions, multilinear pseudo-differential operators and its commutator were studied in \cite{OPR} and \cite{CXY} respectively.

The main results of this paper can be stated as follows. We begin with a sharp weighted inequality in terms of both $\alpha$ and $[\vec{w}]_{A_{\vec{p}}}$.

\begin{theorem}\label{thm:S-sharp}
Let $\alpha\geq 1$ and  $\frac1p=\frac{1}{p_1}+\cdots +\frac{1}{p_m}$ with $1<p_1,\ldots,p_m<\infty$. If $\vec{w} \in A_{\vec{p}}$, then
  \begin{equation}\label{eq:BH}
||S_{\alpha}(\vec{f})||_{L^{p}(\nu_{\vec{w}})} \lesssim  \alpha^{mn} [\vec{w}]_{A_{\vec{p}}}^{\max\{\frac{1}{2},\frac{p_1'}{p},\cdots,\frac{p_m'}{p}\}}\prod_{i=1}^m||f_i||_{L^{p_i}(w_i)},
\end{equation}
where the implicit constant is independent of $\alpha$ and $\vec{w}$. Moreover, \eqref{eq:BH} is sharp in $\alpha$ on all class $A_{\vec{p}}$.
\end{theorem}

In order to present two-weight inequalities for square functions, we give the definition of bump conditions. Given Young functions $A$ and $\vec{B}=(B_1,\ldots,B_m)$, we denote
\begin{equation*}
||(u, \vec{v})||_{A,\, \vec{B},\, \vec{p}} :=
\begin{cases}
\sup\limits_{Q} \|u^{\frac1p}\|_{p, Q} \prod_{j=1}^m \|v_j^{-\frac{1}{p_j}}\|_{B_j, Q}, & \text{if } 1<p \le 2,
\\
\sup\limits_{Q} \|u^{\frac2p}\|_{A, Q}^{\frac12} \prod_{j=1}^m \|v_j^{-\frac{1}{p_j}}\|_{B_j, Q}, & \text{if } 2<p<\infty.
\end{cases}
\end{equation*}

\begin{theorem}\label{thm:bump}
Let $\alpha \ge 1$, $\lambda>2m$, and $\frac1p=\frac{1}{p_1}+\cdots +\frac{1}{p_m}$ with $1<p_1,\dots,p_m<\infty$. If the pair $(u, \vec{v})$ satisfies $||(u, \vec{v})||_{A,\, \vec{B},\, \vec{p}}<\infty$
with $\bar{A} \in B_{(p/2)'}\ (2<p<\infty)$ and $\bar{B_j} \in B_{p_j}$, then
\begin{align}
\label{eq:SLp} \|S_{\alpha}(\vec{f})\|_{L^p(u)}
&\lesssim \alpha^{mn} \mathscr{N}_{\vec{p}} \prod_{j=1}^m \|f_j\|_{L^{p_j}(v_j)},
\\
\label{eq:gLp} \|g_{\lambda}^{*}(\vec{f})\|_{L^p(u)}
&\lesssim \frac{\mathscr{N}_{\vec{p}}}{2^{n(\lambda-2m)}-1} \prod_{j=1}^m \|f_j\|_{L^{p_j}(v_j)},
\end{align}
where
\begin{equation*}
\mathscr{N}_{\vec{p}}  :=
\begin{cases}
||(u, \vec{v})||_{A,\, \vec{B},\, \vec{p}}\prod_{j=1}^m [\bar{B_j}]_{(B_j)_{p_j}}^{\frac{1}{{p_j}}}, & \text{if } 0<p \le 2,
\\
||(u, \vec{v})||_{A,\, \vec{B},\, \vec{p}}[\bar{A}]_{B_{(p/2)'}}^{\frac12-\frac1p} \prod_{j=1}^m [\bar{B_j}]_{(B_j)_{p_j}}^{\frac{1}{{p_j}}}, & \text{if } 2<p<\infty.
\end{cases}
\end{equation*}
\end{theorem}

For arbitrary weights, we have the following Fefferman-Stein inequalities.
\begin{theorem}\label{thm:FS}
Let $\alpha\geq 1$ and $\lambda>2m$. Then for all exponents $\frac1p=\frac{1}{p_1}+\cdots+\frac{1}{p_m}$ with $0<p\leq 2$ and $1<p_1,\ldots, p_m<\infty$, and for all weights $\vec{w}=(w_1, \ldots, w_m)$,
\begin{align}
\label{eq:S-FS} \|S_{\alpha}(\vec{f})\|_{L^p(\nu_{\vec{w}})}
&\lesssim \alpha^{mn}\prod_{i=1}^m \|f_i\|_{L^{p_i}(Mw_i)},
\\
\label{eq:g-FS} \|g_{\lambda}^{*}(\vec{f})\|_{L^p(\nu_{\vec{w}})}
&\lesssim \frac{1}{2^{n(\lambda-2m)}-1} \prod_{i=1}^m \|f_i\|_{L^{p_i}(Mw_i)},
\end{align}
where $\nu_{\vec{w}}=\prod_{i=1}^m w_i^{p/p_i}$.
\end{theorem}

We are going to establish entropy bump estimates. See Section \ref{sec:entropy} for the entropy bump conditions $[\vec{\sigma},\nu]_{\frac{2}{\vec{p}'},\vec{p},\varepsilon,\frac{2}{p},m+1}$ and $[\vec{\sigma},\nu]_{\vec{p},2,\varepsilon}$.

\begin{theorem}\label{thm:entropy}
Let $\alpha\geq 1$, $\lambda>2m$, and let $\frac1p=\frac{1}{p_1}+\cdots+\frac{1}{p_m}$ with $1<p_1,\ldots, p_m<\infty$. Let $\nu$ and $\vec{\sigma}=(\sigma_1,\dots,\sigma_m)$ weights. Assume that $\varepsilon$ is a monotonic increasing function on $(1,\infty)$ satisfying $ \int_{1}^{\infty}\frac{dt}{\varepsilon(t)t}<\infty$. Then,
\begin{align}
\label{eq:S-entr} \|S_{\alpha}(\vec{f \sigma})\|_{L^p(\nu)}
&\lesssim \alpha^{mn} \mathscr{N}_{\vec{p},\varepsilon}\prod_{i=1}^{m}\|f\|_{L^{p_{i}}(\sigma_{i})},
\\
\label{eq:g-entr} \|g_{\lambda}^{*}(\vec{f\sigma})\|_{L^p(\nu)}
&\lesssim \frac{\mathscr{N}_{\vec{p},\varepsilon}}{2^{n(\lambda-2m)}-1} \prod_{i=1}^{m}\|f\|_{L^{p_{i}}(\sigma_{i})},
\end{align}
where
\begin{equation*}
\mathscr{N}_{\vec{p},\varepsilon}  :=
\begin{cases}
{}\lfloor \vec{\sigma},\nu \rfloor_{\frac{2}{\vec{p}'},\vec{p},\varepsilon,\frac{2}{p},m+1}^{\frac{1}{2}}, & \text{if } 0<p \le 2,
\\
{} \lfloor \vec{\sigma},\nu \rfloor_{\vec{p},2,\varepsilon}^{\frac{1}{p}}, & \text{if } 2<p<\infty.
\end{cases}
\end{equation*}
with $\vec{p}=(p_1,\dots,p_m,p')$ and $\frac{2}{\vec{p}'}=\big(\frac{2}{p'_1},\ldots,\frac{2}{p'_m},\frac2p\big)$.
\end{theorem}

Next, we turn to the weak type estimates for Littlewood-Paley operators.
\begin{theorem}\label{thm:mixed}
Let $\alpha\geq 1$ and $\lambda>2m$. Let $\vec{w}=(w_1,\ldots,w_m)$ and $u=\prod_{i=1}^m w_i^{1/m}$. If $\vec{w}$ and $v$ satisfy
\begin{align*}
(1) \quad \vec{w} \in A_{\vec{1}} \text{ and } u v^{1/m} \in A_{\infty},\ \
\text{ or }\quad  (2)\quad w_1,\ldots,w_m \in A_1 \text{ and } v \in A_{\infty},
\end{align*}
then we have
\begin{align}
\label{eq:Smixed} \bigg\|\frac{S_{\alpha}(\vec{f})}{v}\bigg\|_{L^{1/m,\infty}(u v^{1/m})}
&\lesssim \prod_{i=1}^m \| f_i \|_{L^1(w_i)},
\\
\label{eq:gmixed} \bigg\|\frac{g_{\lambda}^{*}(\vec{f})}{v}\bigg\|_{L^{1/m,\infty}(u v^{1/m})} &\lesssim \prod_{i=1}^m \| f_i \|_{L^1(w_i)}.
\end{align}
In particular, both $S_{\alpha}$ and $g_{\lambda}^{*}$ are bounded from $L^1(w_1) \times \cdots \times L^1(w_m)$ to $L^{1/m, \infty}(\nu_{\vec{w}})$ for every $\vec{w} \in A_{\vec{1}}$.
\end{theorem}

\begin{theorem}\label{thm:local}
Let $\alpha\geq 1$ and $\lambda>2m$. Let $Q$ be a cube and every function $f_j \in L^{\infty}_c(\Rn)$ with $\supp(f_j) \subset Q$, $j=1,\ldots,m$. Then there exist constants $c_1>0$ and $c_2>0$ such that
\begin{align}
\label{eq:local-1} \big| \big\{x \in Q:  S_{\alpha}(\vec{f})(x) > t \M(\vec{f})(x)  \big\}\big|
&\leq c_1 e^{- c_2 \beta_1 t^2}  |Q|,
\\
\label{eq:local-2} \big| \big\{x \in Q:  g_{\lambda}^{*}(\vec{f})(x) > t \M(\vec{f})(x)  \big\}\big|
&\leq c_1 e^{- c_2 \beta_2 t^2}  |Q|,
\end{align}
for all $t>0$, where $\beta_1=\alpha^{-2mn}$ and $\beta_2=(1-2^{-n(\lambda-2m)/2})^2$.
\end{theorem}

\section{Preliminaries}

\subsection{Multiple weights}
The multilinear maximal operators $\mathcal{M}$ are defined by
$$\mathcal{M}(\vec{f})(x)=\sup_{Q\ni x} \prod_{j=1}^m  \fint_Q |f_j(y_j)|dy_j,$$
where the supremum is taken over all the cubes containing $x$. The corresponding theory of weights for this new maximal function  gives the right class of multiple weights for multilinear Calder\'{o}n-Zygmund operators.
\begin{definition}\label{def3}
Let $1\leq p_1,\ldots,p_m<\infty.$ Given a vector of weights $\vec{w}=(w_1,\cdots, w_m)$, we say that $\vec{w} \in A_{\vec{p}}$ if
$$[\vec{w}]_{A_{\vec{p}}} := \sup_Q \bigg(\fint_Q \vec{w}\, dx \bigg)^{\frac{1}{p}}\prod_{i=1}^m
\bigg(\fint_Q w_i^{1-p'_i} dx \bigg)^{\frac{1}{p'_i}}< \infty,$$
where $\frac1p=\frac{1}{p_1}+\cdots+\frac{1}{p_m}$ and $\nu_{\vec{w}}=\prod_{i=1}^m w_i^{p/p_i}$.
When $p_i=1,$ $\big(\fint_Q w_i^{1-p'_i} dx\big)^{1/p'_i}$ is understood as $(\inf_Q w_i)^{-1}.$
\end{definition}

The characterizations of multiple weights were given in \cite{LOPTT} and \cite{CWX}.

\begin{lemma}\label{lem:weight}
Let $\frac{1}{p} = \frac{1}{p_{1}} + \cdots + \frac{1}{p_{m}}$ with $1 \leq p_1, \ldots, p_m < \infty$,  and $p_0=\min\{p_i\}_i$. Then the following statements hold :
\begin{enumerate}
\item[$(1)$] $ A_{r_1 \vec{p}} \subsetneq A_{r_2 \vec{p}},
\ \ \text{for any} \ \ 1/p_0 \leq r_1 < r_2 < \infty.$
\item[$(2)$] $A_{\vec{p}}=\bigcup_{1/p_0 \leq r < 1} A_{r \vec{p}}$.
\item[$(3)$] $\vec{w} \in A_{\vec{p}}$ if and only if $\nu_{\vec{w}} \in A_{mp}$ and $w_i^{1-p_i'} \in A_{mp_i'}$, $i=1,\ldots,m$. Here, if $p_i=1$, $w_i^{1-p_i'} \in A_{mp_i'}$ is understood as $w_i^{1/m} \in A_{1}$.
\end{enumerate}
\end{lemma}

\subsection{Dyadic cubes}
Denote by $\ell(Q)$ the sidelength of the cube $Q$. Given a cube $Q_0 \subset \Rn$, let $\D(Q_0)$ denote the set of all dyadic cubes with respect to $Q_0$, that is, the cubes obtained by repeated subdivision of $Q_0$ and each of its descendants into $2^n$ congruent subcubes.

\begin{definition}
A collection $\D$ of cubes is said to be a dyadic grid if it satisfies
\begin{enumerate}
\item [(1)] For any $Q \in \D$, $\ell(Q) = 2^k$ for some $k \in \Z$.
\item [(2)] For any $Q,Q' \in \D$, $Q \cap Q' = \{Q,Q',\emptyset\}$.
\item [(3)] The family $\D_k=\{Q \in \D; \ell(Q)=2^k\}$ forms a partition of $\Rn$ for any $k \in \Z$.
\end{enumerate}
\end{definition}

\begin{definition}
A subset $\S$ of a dyadic grid is said to be $\eta$-sparse, $0<\eta<1$, if for every $Q \in \S$, there exists a measurable set $E_Q \subset Q$ such that $|E_Q| \geq \eta |Q|$, and the sets $\{E_Q\}_{Q \in \S}$ are pairwise disjoint.
\end{definition}

By a median value of a measurable function $f$ on a cube $Q$ we mean a possibly non-unique, real number $m_f (Q)$ such that
\[
\max \big\{|\{x \in Q : f(x) > m_f(Q) \}|,
|\{x \in Q : f(x) < m_f(Q) \}| \big\} \leq |Q|/2.
\]
The decreasing rearrangement of a measurable function $f$ on $\Rn$ is
defined by
\[
f^*(t) = \inf \{ \alpha > 0 : |\{x \in \Rn : |f(x)| > \alpha \}| < t \},
\quad 0 < t < \infty.
\]
The local mean oscillation of $f$ is
\[
\omega_{\lambda}(f; Q)
= \inf_{c \in \R} \big( (f-c) \mathbf{1}_{Q} \big)^* (\lambda |Q|),
\quad 0 < \lambda < 1.
\]
Given a cube $Q_0$, the local sharp maximal function is
defined by
\[
M_{\lambda; Q_0}^{\sharp} f (x)
= \sup_{x \in Q \subset Q_0} \omega_{\lambda}(f; Q).
\]

Observe that for any $\delta > 0$ and $0 < \lambda < 1$
\begin{equation}\label{e:mfQ}
|m_f(Q)| \leq (f \mathbf{1}_Q)^* (|Q|/2) \ \ \text{and} \ \
(f \mathbf{1}_Q)^* (\lambda |Q|) \leq
\left( \frac{1}{\lambda |Q|} \int_{Q} |f|^{\delta} dx \right)^{1/{\delta}}.
\end{equation}
The following theorem was proved by Hyt\"{o}nen \cite[Theorem~2.3]{Hy2} in order to improve Lerner's formula given in \cite{Ler11} by getting rid of the local sharp maximal function.

\begin{lemma}
Let $f$ be a measurable function on $\Rn$ and let $Q_0$ be a fixed cube. Then there exists a (possibly empty) sparse family $\S(Q_0) \subset \D(Q_0)$ such that
\begin{equation}\label{eq:mf}
|f (x) - m_f (Q_0)| \leq 2 \sum_{Q \in \S(Q_0)} \omega_{2^{-n-2}}(f; Q) \mathbf{1}_Q (x), \quad a. e. ~ x \in Q_0.
\end{equation}
\end{lemma}

\subsection{Orlicz maximal operators}
A function $\Phi:[0,\infty) \to [0,\infty)$ is called a Young function if it is continuous, convex, strictly increasing, and satisfies
\begin{equation*}
\lim_{t\to 0^{+}}\frac{\Phi(t)}{t}=0 \quad\text{and}\quad \lim_{t\to\infty}\frac{\Phi(t)}{t}=\infty.
\end{equation*}
Given $p \in[1, \infty)$, we say that a Young function $\Phi$ is a  $p$-Young function, if $\Psi(t)=\Phi(t^{1/p})$ is a Young function.

If $A$ and $B$ are Young functions, we write $A(t) \simeq B(t)$ if there are constants $c_1, c_2>0$ such that
$c_1 A(t) \leq B(t) \leq c_2 A(t)$ for all $t \geq t_0>0$. Also, we denote $A(t) \preceq B(t)$ if there exists $c>0$ such that $A(t) \leq B(ct)$ for all $t \geq t_0>0$. Note that for all Young functions $\phi$, $t \preceq \phi(t)$. Further, if $A(t)\leq cB(t)$ for some $c>1$, then by convexity, $A(t) \leq B(ct)$.

A function $\Phi$ is said to be doubling, or $\Phi \in \Delta_2$, if there is a constant $C>0$ such that $\Phi(2t) \leq C \Phi(t)$ for any $t>0$. Given a Young function $\Phi$, its complementary function $\bar{\Phi}:[0,\infty) \to [0,\infty)$ is defined by
\[
\bar{\Phi}(t):=\sup_{s>0}\{st-\Phi(s)\}, \quad t>0,
\]
which clearly implies that
\begin{align}\label{eq:stst}
st \leq \Phi(s) + \bar{\Phi}(t), \quad s, t > 0.
\end{align}
Moreover, one can check that $\bar{\Phi}$ is also a Young function and
\begin{equation}\label{eq:Young-1}
t \leq \Phi^{-1}(t) \bar{\Phi}^{-1}(t) \leq 2t, \qquad t>0.
\end{equation}
In turn, by replacing $t$ by $\Phi(t)$ in first inequality of \eqref{eq:Young-1}, we obtain
\begin{equation}\label{eq:Young-2}
\bar{\Phi} \Big(\frac{\Phi(t)}{t}\Big) \leq \Phi(t), \qquad t>0.
\end{equation}

Given a Young function $\Phi$, we define the Orlicz space $L^{\Phi}(\Omega, \mu)$ to be the function space with Luxemburg norm
\begin{align}\label{eq:Orlicz}
\|f\|_{L^{\Phi}(\Omega, u)} := \inf\bigg\{\lambda>0:
\int_{\Omega} \Phi \Big(\frac{|f(x)|}{\lambda}\Big) d\mu(x) \leq 1 \bigg\}.
\end{align}
Now we define the Orlicz maximal operator
\begin{align*}
M_{\Phi}f(x) := \sup_{Q \ni x} \|f\|_{\Phi, Q} := \sup_{Q \ni x} \|f\|_{L^{\Phi}(Q, \frac{dx}{|Q|})},
\end{align*}
where the supremum is taken over all cubes $Q$ in $\Rn$. When $\Phi(t)=t^p$, $1\leq p<\infty$,
\begin{align*}
\|f\|_{\Phi, Q} = \bigg(\fint_{Q} |f(x)|^p dx \bigg)^{\frac1p}=:\|f\|_{p, Q}.
\end{align*}
In this case, if $p=1$, $M_{\Phi}$ agrees with the classical Hardy-Littlewood maximal operator $M$; if $p>1$, $M_{\Phi}f=M_pf:=M(|f|^p)^{1/p}$. If $\Phi(t) \preceq \Psi(t)$, then $M_{\Phi}f(x) \leq c M_{\Psi}f(x)$ for all $x \in \Rn$.

The H\"{o}lder inequality can be generalized to the scale of Orlicz spaces \cite[Lemma~5.2]{CMP11}.
\begin{lemma}
Given a Young function $A$, then for all cubes $Q$,
\begin{equation}\label{eq:Holder-AA}
\fint_{Q} |fg| dx \leq 2 \|f\|_{A, Q} \|g\|_{\bar{A}, Q}.
\end{equation}
More generally, if $A$, $B$ and $C$ are Young functions such that $A^{-1}(t) B^{-1}(t) \leq c_1 C^{-1}(t), $ for all $t \geq t_0>0$,
then
\begin{align}\label{eq:Holder-ABC}
\|fg\|_{C, Q} \leq c_2 \|f\|_{A, Q} \|g\|_{B, Q}.
\end{align}
\end{lemma}

 The following result is an extension of the well-known Coifman-Rochberg theorem. The proof can be found in \cite[Lemma~4.2]{HP}.
\begin{lemma}
Let $\Phi$ be a Young function and $w$ be a nonnegative function such that $M_{\Phi}w(x)<\infty$ a.e.. Then
 \begin{align}
 \label{eq:CR-Phi} [(M_{\Phi}w)^{\delta}]_{A_1} &\le c_{n,\delta}, \quad\forall \delta \in (0, 1),
 \\
\label{eq:MPhiRH} [(M_{\Phi} w)^{-\lambda}]_{RH_{\infty}} &\le c_{n,\lambda},\quad\forall \lambda>0.
 \end{align}
 \end{lemma}

Given $p \in (1, \infty)$, a Young function $\Phi$ is said to satisfy the $B_p$ condition (or, $\Phi \in B_p$) if for some $c>0$,
\begin{align}\label{def:Bp}
\int_{c}^{\infty} \frac{\Phi(t)}{t^p} \frac{dt}{t} < \infty.
\end{align}
Observe that if \eqref{def:Bp} is finite for some $c>0$, then it is finite for every $c>0$. Let $[\Phi]_{B_p}$ denote the value if $c=1$ in \eqref{def:Bp}. It was shown in \cite[Proposition~5.10]{CMP11} that if $\Phi$ and $\bar{\Phi}$ are doubling Young functions, then $\Phi \in B_p$ if and only if
\begin{align*}
\int_{c}^{\infty} \bigg(\frac{t^{p'}}{\bar{\Phi}(t)}\bigg)^{p-1} \frac{dt}{t} < \infty.
\end{align*}

Let us present two types of $B_p$ bumps. An important special case is the ``log-bumps" of the form
\begin{align}\label{eq:log}
A(t) =t^p \log(e+t)^{p-1+\delta}, \quad  B(t) =t^{p'} \log(e+t)^{p'-1+\delta},\quad \delta>0.
\end{align}
Another interesting example is the ``loglog-bumps" as follows:
\begin{align}
\label{eq:loglog-1} &A(t)=t^p \log(e+t)^{p-1} \log\log(e^e+t)^{p-1+\delta}, \quad \delta>0\\
\label{eq:loglog-2} &B(t)=t^{p'} \log(e+t)^{p'-1} \log\log(e^e+t)^{p'-1+\delta}, \quad \delta>0.
\end{align}
Then one can verify that in both cases above, $\bar{A} \in B_{p'}$ and $\bar{B} \in B_p$ for any $1<p<\infty$.

The $B_p$ condition can be also characterized by the boundedness of the Orlicz maximal operator $M_{\Phi}$. Indeed, the following result was given in \cite[Theorem~5.13]{CMP11} and \cite[eq. (25)]{HP}.
\begin{lemma}\label{lem:MBp}
Let $1<p<\infty$. Then $M_{\Phi}$ is bounded on $L^p(\Rn)$ if and only if $\Phi \in B_p$. Moreover, $\|M_{\Phi}\|_{L^p(\Rn) \to L^p(\Rn)} \le C_{n,p} [\Phi]_{B_p}^{\frac1p}$. In particular, if the Young function $A$ is the same as the first one in \eqref{eq:log} or \eqref{eq:loglog-1}, then
\begin{equation}\label{eq:MAnorm}
\|M_{\bar{A}}\|_{L^{p'}(\Rn) \to L^{p'}(\Rn)} \le c_n p^2 \delta^{-\frac{1}{p'}},\quad\forall \delta \in (0, 1].
\end{equation}
\end{lemma}

\begin{definition}\label{def:sepbum}
Given $p \in (1, \infty)$, let $A$ and $B$ be Young functions such that $\bar{A} \in B_{p'}$ and $\bar{B} \in B_p$. We say that the pair of weights $(u, v)$ satisfies the {\tt double bump condition} with respect to $A$ and $B$ if
\begin{align}\label{eq:uvABp}
[u, v]_{A,B,p}:=\sup_{Q} \|u^{\frac1p}\|_{A,Q} \|v^{-\frac1p}\|_{B,Q} < \infty.
\end{align}
where the supremum is taken over all cubes $Q$ in $\Rn$. Also, $(u, v)$ is said to satisfy the separated bump condition if
\begin{align}
\label{eq:uvAp} [u, v]_{A,p'} &:= \sup_{Q} \|u^{\frac1p}\|_{A,Q} \|v^{-\frac1p}\|_{p',Q} < \infty,
\\
\label{eq:uvpB} [u, v]_{p,B} &:= \sup_{Q} \|u^{\frac1p}\|_{p,Q} \|v^{-\frac1p}\|_{B,Q} < \infty.
\end{align}
\end{definition}

Note that if $A(t)=t^p$ in \eqref{eq:uvAp} or $B(t)=t^p$ in \eqref{eq:uvpB}, each of them actually is two-weight $A_p$ condition and we denote them by $[u, v]_{A_p}:=[u, v]_{p,p'}$.  Also, the separated bump condition is weaker than the double bump condition. Indeed, \eqref{eq:uvABp} implies \eqref{eq:uvAp} and \eqref{eq:uvpB}, but the reverse direction is incorrect.
The first fact holds since $\bar{A} \in B_{p'}$ and $\bar{B} \in B_p$ respectively indicate $A$ is a $p$-Young function and $B$ is a $p'$-Young function. The second fact was shown in \cite[Section~7]{ACM} by constructing log-bumps.

\begin{lemma}\label{lem:M-uv}
Let $1<p<\infty$, let $A$, $B$ and $\Phi$ be Young functions such that $A \in B_p$ and $A^{-1}(t)B^{-1}(t) \lesssim \Phi^{-1}(t)$ for any $t>t_0>0$. If a pair of weights $(u, v)$ satisfies $[u, v]_{p, B}<\infty$, then
\begin{align}\label{eq:MPhi-uv}
\|M_{\Phi}f\|_{L^p(u)} \leq C [u, v]_{p, B} [A]_{B_p}^{\frac1p} \|f\|_{L^p(v)}.
\end{align}
Moreover, \eqref{eq:MPhi-uv} holds for $\Phi(t)=t$ and $B=\bar{A}$ satisfying the same hypotheses.  In this case, $\bar{A} \in B_p$ is necessary.
\end{lemma}

The two-weight inequality above was established in \cite[Theorem~5.14]{CMP11} and \cite[Theorem~3.1]{CP99}.  The weak type inequality for $M_{\Phi}$ was also obtained in \cite[Proposition~5.16]{CMP11} as follows.

\begin{lemma}\label{lem:Muv-weak}
Let $1<p<\infty$, let $B$ and $\Phi$ be Young functions such that $t^{\frac1p} B^{-1}(t) \lesssim \Phi^{-1}(t)$ for any $t>t_0>0$. If a pair of weights $(u, v)$ satisfies $[u, v]_{p, B}<\infty$, then
\begin{align}\label{eq:MPuv}
\|M_{\Phi}f\|_{L^{p,\infty}(u)} \leq C \|f\|_{L^p(v)}.
\end{align}
Moreover, \eqref{eq:MPuv} holds for $M$ if and only if $[u, v]_{A_p}<\infty$.
\end{lemma}

\section{Sharpness in aperture $\alpha$}

The goal of this section is to give the proof of Theorem \ref{thm:S-sharp}. To this end, we establish some fundamental estimates.
\begin{lemma}\label{lem:cont}
$\psi(x,\vec{y})$ is continuous at $(x_0,y_{1,0},\dots, y_{m,0})$ with $x_0 \neq y_{j, 0}$, $j=1,2,\dots,m.$
\end{lemma}

\begin{proof}
Let $x_0\neq y_{j, 0}$ for $j=1,2,\dots,m,$ and let
\begin{align*}
|x-x_0|<\frac{1}{4} \min_{1\leq i \leq m}\{|x_0-y_{i,0}|\}, \qquad
|y_j-y_{j,0}|<\frac{1}{2} \min_{1\leq i \leq m}\{|x_0-y_{i,0}|\}.
\end{align*}
Then we get
\begin{equation*}
|y_j-y_{j,0}|< \frac{1}{2}|x_0-y_{j,0}|
\end{equation*}
and
\begin{equation*}
|x_0-y_{j,0}|\leq |x_0-y_j|+|y_j-y_{j,0}|<|x_0-y_j|+\frac{1}{2}|x_0-y_{j,0}|
\end{equation*}
and so
\begin{equation*}
|x_0-y_{j,0}|< 2|x_0-y_j|, \quad j=1,\dots,m,
\end{equation*}
which implies
\begin{equation*}
|x-x_0|<\frac{1}{4}|x_0-y_{j,0}|<\frac{1}{2}|x_0-y_{j}|,\quad j=1,\dots,m,
\end{equation*}
Therefore, we have
\begin{align*}
&|\psi(x,y_1,\dots,y_m)-\psi(x_0,y_{1,0},\dots,y_{m,0})|
\\
&\leq |\psi(x,y_1,\dots,y_m)-\psi(x_0,y_1,\dots,y_m)|
\\
&\qquad+\sum_{j=1}^m |\psi(x_0,y_{1,0},\dots,y_{j-1, 0},y_j,y_{j+1},\dots, y_m)
\\
&\qquad\qquad\quad-\psi(x_0,y_{1,0},\dots,y_{j-1, 0},y_{j,0},y_{j+1},\dots, y_m)|
\\
&\leq \frac{A|x-x_0|^\gamma}{(1+\sum_{i=1}^m |x-y_i|)^{mn+\delta+\gamma}}
+\sum_{j=1}^m \frac{A|y_j-y_{j,0}|^\gamma}{(1+\sum_{i=1}^m |x_0-y_i|)^{mn+\delta+\gamma}}.
\end{align*}
This shows  $\psi(x, \vec{y})$ is continuous at $(x_0,y_{1,0},\dots, y_{m,0})\in \mathbb{R}^{n(m+1)}$ with $x_0 \neq y_{j,0}$, $j=1,2,\dots,m.$
\end{proof}

\begin{lemma}\label{lem:AA}
There exist $x_0 \in \Rn$, $r_0>0$, $t_0>1$ and $f_j\in \S(\Rn)$, $j=1,\ldots,m$, such that
\begin{equation}\label{eq:AA}
A_0 := \iint_{\Omega_0} |\psi_t(\vec{f})(y)|^2dydt \in (0, \infty),
\end{equation}
where $\Omega_0:=B(0, |x_0|+r_0) \times [1, t_0]$.
\end{lemma}

\begin{proof}
Since $\psi$ is a non-zero function in $\R^{n(m+1)}$, there exist $x_0,y_{1,0},\dots,y_{m,0} \in \Rn$ such that $x_0\neq y_{i,0}(i=1,\dots,m)$ and $\psi(x_0,y_{1,0},\dots,y_{m,0})\neq 0$.  By Lemma \ref{lem:cont}, there exists $r_0>0$ such that $\psi(x,\vec{y})>0$ or $\psi(x,\vec{y})<0$ for all $x\in B(x_0,r_0)$ and $y_j\in B(y_{j,0},r_0)$, $j=1,\dots,m$. Keeping these notation in mind, we set
\begin{equation*}
t_0=\bigg(1-\frac{r_0}{2\max\{|x_0|, |y_{1, 0}|,\ldots, |y_{m,0}|\}}\bigg)^{-1},
\end{equation*}
if  $\max\{|x_0|, |y_{1,0}|,\dots,|y_{m,0}|\} \ge r_0$, and $t_0=2$ otherwise. Then, one has
\begin{align}\label{eq:xt}
\Big|\frac{x}{t}-x_0\Big|<r_0 \quad\text{and}\quad \Big|\frac{y_i}{t}-y_{i, 0}\Big|<r_0, \quad i=1,\ldots,m,
\end{align}
for all $1<t<t_0$, $|x-x_0|<\frac{r_0}{2}$ and $|y_i-y_{i, 0}|<\frac{r_0}{2}$, $i=1,\ldots,m$. Indeed,  if  $\max\{|x_0|, |y_{1,0}|,\dots,|y_{m,0}|\} < r_0$, it follows $|\frac{x}{t}-x_0|<\frac{|x-x_0|}{t}+(1-\frac{1}{t})|x_0|<|x-x_0|+\frac{|x_0|}{2}<r_0,$ and similarly we get $|\frac{y_i}{t}-y_{i,0}|<r_0$, $i=1,\ldots,m$. In the case $y_{j_0, 0}:=\max\{|x_0|, |y_{1,0}|,\dots,|y_{m,0}|\} \ge r_0$,  we have
\begin{align*}
\frac{|y-y_{i,0}|}{t} < |y-y_{j_0,0}| < \Big(1-\frac{1}{t_0}\Big)|y_{j_0,0}|
=\bigg(1-\Big(1-\frac{r_0}{2|y_{j_0,0}|}\Big)\bigg)|y_{j_0,0}| = \frac{r_0}{2}.
\end{align*}
These show that
\begin{align*}
\Big|\frac{y_i}{t}-y_{i,0}\Big|<r_0,\quad i=1,\ldots,m.
\end{align*}
Similarly, we get $|\frac{x}{t}-x_0|<r_0$. This shows \eqref{eq:xt}.

We may assume that $\psi(x, \vec{y})>0$ for $1<t<t_0$ and $|x-x_0|<\frac{r_0}{2}$, $|y_i-y_{i,0}|<\frac{r_0}{2}$,  $i=1,\ldots,m$. Pick $f_j\in \S(\Rn)$, $j=1,\ldots,m$ such that $\supp f_j \in B(y_{j,0},\frac{r_0}{2})$ and $f_j(y_j)>0$ for $y_j\in B(y_{j,0}, \frac{r_0}{4})$, $j=1,2,\dots,m$. Then it follows from \eqref{eq:xt} that
\begin{align*}
\psi_t(\vec{f})(x) =\frac{1}{t^{mn}} \int_{ B(y_{1,0},\frac{r_0}{2})\times \dots \times  B(y_{m,0},\frac{r_0}{2})} \psi(\frac{x}{t},\frac{y_1}{t},\dots,\frac{y_m}{t}) \prod_{j=1}^m f_j(y_j)d\vec{y} > 0
\end{align*}
for all $1<t<t_0$ and $|x-x_0|>\frac{r_0}{2}$. In particular,
\begin{align}\label{eq:postive}
A_0=\iint_{B(0, |x_0|+r_0) \times [1, t_0]}\big| \psi_t(\vec{f})(y)\big|^2dydt>0.
\end{align}

On the other hand, by using the size condition of $\psi$, we obtain for every $(y, t) \in \Omega_0$,
\begin{align*}
|\psi_t(\vec{f})(y)| & \le \frac{1}{t^{mn}}\int_{\mathbb{R}^{mn}}
\bigg|\psi\Big(\frac{y}{t}, \frac{y_1}{t},\cdots,\frac{y_m}{t}\Big)\bigg| \prod_{j=1}^m |f_j(y_j)|dy_j
\\
&\le \frac{1}{t^{mn}}\int_{\mathbb{R}^{mn}}
\frac{\prod_{j=1}^m |f_j(y_j)|dy_j}{\left(1+\frac{|y-y_1|}{t}+\cdots + \frac{|y-y_m|}{t}\right)^{mn+\delta}}
\le \frac{\prod_{j=1}^m ||f_j||_{L^1}}{t^{mn}}.
\end{align*}
This immediately yields that
\begin{equation}\label{eq:finite}
A_0 \le  \iint_{\Omega_0} \frac{\prod_{j=1}^m ||f_j||^2_{L^1}}{t^{2mn}}dydt
\lesssim \prod_{j=1}^m ||f_j||^2_{L^1}< \infty.
\end{equation}
Consequently, the desired result follows from \eqref{eq:postive} and \eqref{eq:finite}.
\end{proof}

\begin{lemma}\label{unbounded}
Let $0<\lambda<2m$ and $\frac{1}{m}<p<\frac{2}{\lambda}$. Then $g^*_{\lambda}$ is not bounded from $L^{p_1}\times \cdots \times L^{p_m}$ to $L^p$, where $\frac{1}{p}=\frac{1}{p_1}+\cdots +\frac{1}{p_m}$ with $1\le p_1,\ldots, p_m<\infty$.
\end{lemma}

\begin{proof}
By Lemma \ref{lem:AA}, there exist $x_0 \in \Rn$, $r_0>0$, $t_0>1$ and $f_j\in \S(\Rn)$, $j=1,\ldots,m$, such that
$0<A_0<\infty$, where $A_0$ is defined in \eqref{eq:AA}. Write $R_0:=2(|x_0|+r_0+t_0)$. Then for all $|x|>R_0$ and $(y, t) \in \Omega_0$,
\begin{align*}
\frac{|x|}{2}< |x|-|y|\le t+|x-y|\le |x|+|y|+t_0 \le 2|x|.
\end{align*}
Thus,  $t+|x-y| \simeq |x|.$ This gives that for all $|x|>R_0$,
\begin{align*}
g^*_{\lambda}(\vec{f})(x)^2
&\geq \iint_{\Omega_0}  \frac{t^{n\lambda-n-1}}{|x|^{n\lambda}} |\psi_t(\vec{f})(y)|^2 dydt
\\
&\gtrsim \frac{1}{|x|^{n\lambda}} \iint_{\Omega_0} |\psi_t(\vec{f})(y)|^2dydt
=\frac{A_0}{|x|^{n\lambda}}.
\end{align*}
Therefore, for any $\lambda\le \frac{2}{p}$,
\begin{align*}
\|g^*_{\lambda}(\vec{f})\|_{L^p}^p \gtrsim A_0^{\frac{p}{2}} \int_{|x|>R_0} \frac{dx}{|x|^{\frac{n\lambda p}{2}}}=\infty.
\end{align*}
On the other hand, for  $\vec{f} \in \S(\Rn)\times \cdots \times \S(\Rn)$, we have $\prod_{j=1}^m ||f_j||_{L^{p_j}}<\infty$.  As a consequence,  $g^*_{\lambda}$ is not bounded from  $L^{p_1}\times \cdots \times L^{p_m}$ to $L^p$ whenever $\lambda\le \frac{2}{p}$.

Finally, for $0<\lambda<2m$ (equivalently $\frac{1}{m}<\frac{2}{\lambda}$), one can choose $p\in (\frac1m,\frac{2}{\lambda})$, $1<p_1,\ldots,p_m<\infty$ with $\frac{1}{p}=\frac{1}{p_1}+\cdots +\frac{1}{p_m}$ such that $g^*_{\lambda}$ is not bounded from  $L^{p_1}\times \cdots \times L^{p_m}$ to $L^p$.
\end{proof}

\begin{proof}[\textbf{Proof of Theorem \ref{thm:S-sharp}.}]
It follows from \cite{BH} that
\begin{equation}\label{eq:Ssharp}
\|S_{\alpha}(\vec{f})\|_{L^p(\nu_{\vec{w}})}
\lesssim \alpha^{mn} [\vec{w}]_{A_{\vec{p}}}^{\max\{\frac{1}{2},\frac{p_1'}{p},\cdots,
\frac{p_m'}{p}\}}\prod_{i=1}^{m}\|f_i\|_{L^{p_i}(w_i)},
\end{equation}
for all $\frac{1}{p}=\frac{1}{p_1}+\cdots +\frac{1}{p_m}$ with $1<p_1,\ldots,p_m<\infty$, and for all $\vec{w} \in A_{\vec{p}}$, where the implicit constant is independent of $\alpha$ and $\vec{w}$. Now, we seek for $\gamma(\alpha)=\alpha^r$  such that
\begin{equation*}
||S_{\alpha}(\vec{f})||_{L^{p}(\nu_{\vec{w}})}
\lesssim \gamma(\alpha) [\vec{w}]_{A_{\vec{p}}}^{\max\{\frac{1}{2},\frac{p_1'}{p},\cdots,\frac{p_m'}{p}\}}\prod_{i=1}^{m}\|f_i\|_{L^{p_i}(w_i)}.
\end{equation*}
We follow Lerner's idea to show $r\geq mn$ for any $1/m<p<\infty.$
In fact, for the case $r<mn$ we can reach a contradiction as follows.
This means that the power growth $\gamma(\alpha)=\alpha^{mn}$ in
\eqref{eq:Ssharp} is sharp.

Using the standard estimate
\begin{equation}\label{eq:gSS}
 g^*_{\lambda}(\vec{f})(x) \le S_{1}(\vec{f})(x)+\sum_{k=0}^{\infty} 2^{-\frac{k\lambda n} {2}}S_{2^{k+1}}(\vec{f})(x),
\end{equation}
we get for some fixed $\frac{1}{q}=\frac{1}{q_1}+\cdots+\frac{1}{q_m}$ with $1<q_1,\ldots,q_m<\infty$, and $\gamma(\alpha)=\alpha^{r_0}$
\begin{equation*}
\|g^*_{\lambda}(\vec{f})\|_{L^{q}(\nu_{\vec{w}})}
\lesssim \bigg(\sum_{k=0}^{\infty} 2^{-\frac{k\lambda n} {2}}  2^{kr_0}\bigg)
[\vec{w}]_{A_{\vec{q}}}^{\max\{\frac{1}{2},\frac{q_1'}{q},\cdots,\frac{q_m'}{q}\}}\prod_{i=1}^{m}\|f_i\|_{L^{q_i}(w_i)}.
\end{equation*}
This means that if $\lambda> \frac{2r_0}{n}$,  $g^*_{\lambda}$ is bounded from $L^{q_1}(w_1)\times \cdots \times L^{q_m}(w_m)$ to $L^q(\nu_{\vec{w}})$. From this, by extrapolation(see \cite{LMO}), we get  $g^*_{\lambda}$ is bounded from $L^{p_1}\times \cdots \times L^{p_m}$ to $L^p$ for any $p>1/m$, whenever $\lambda> \frac{2r_0}{n}$. But by Lemma \ref{unbounded}, we know  $g^*_{\lambda}$ is not bounded from $L^{p_1}\times \cdots \times L^{p_m}$ to $L^p$ for $\lambda<2m$ and $\frac{1}{m}<p<\frac{2}{\lambda}$. If $r_0<mn,$ we would obtain a contradiction to the latter fact for $p$ sufficiently close to $1/m$.
\end{proof}

\section{Bump and Fefferman-Stein inequalities}
In this section, we will prove bump inequalities (Theorem \ref{thm:bump}) and Fefferman-Stein inequalities (Theorem \ref{thm:FS}). Our strategy is to use the sparse domination for the multilinear Littlewood-Paley operators.

\begin{proof}[\textbf{Proof of Theorem \ref{thm:bump}.}]
Given $r \ge 1$ and a sparse family $\S$, we denote
\begin{align*}
\A_{\S}^r(\vec{f})(x) := \bigg(\sum_{Q \in \S} \prod_{i=1}^m \langle f_i \rangle_{Q}^r \mathbf{1}_Q(x)\bigg)^{\frac1r}.
\end{align*}
The sparse domination below will provide us great convenience:
\begin{align}
\label{eq:S-sparse} S_{\alpha}\vec{f}(x) &\le c_n \alpha^{mn} \sum_{j=1}^{3^n} \A_{\S_j}^2 (|\vec{f}|)(x) ,\quad \text{a.e. } x \in \Rn,
\\
\label{eq:g-sparse} g_{\lambda}^{*}\vec{f}(x) &\le \frac{c_n}{2^{n(\lambda-2m)}-1} \sum_{j=1}^{3^n} \A_{\S_j}^2 (|\vec{f}|)(x) ,\quad \text{a.e. } x \in \Rn,
\end{align}
where $\S_j$ is a sparse family for each $j=1,\ldots,3^n$. These results are explicitly proved in \cite{BH}. By \eqref{eq:S-sparse} and \eqref{eq:g-sparse}, the inequalities \eqref{eq:SLp} and \eqref{eq:gLp} follow from the following
\begin{align}\label{eq:ASLp}
\|\A_{\S}^2(\vec{f})\|_{L^p(u)} \lesssim \mathscr{N}_p \prod_{j=1}^m \|f_j\|_{L^{p_j}(v_j)},
\end{align}
for every sparse family $\S$, where the implicit constant does not depend on $\S$.

To show \eqref{eq:ASLp}, we begin with the case $1<p \le 2$. Actually,
the H\"{o}lder inequality \eqref{eq:Holder-AA} gives that
\begin{align}\label{eq:ASp1}
\|\A_{\S}^2(\vec{f})\|_{L^p(u)}^p
&=\int_{\Rn} \bigg(\sum_{Q \in \S} \prod_{j=1}^m \langle f_j \rangle_Q^2 \mathbf{1}_{Q}(x)\bigg)^{\frac{p}{2}} u(x) dx
\le \sum_{Q \in \S} \prod_{j=1}^m\langle |f_j| \rangle_Q^p u(Q)
\nonumber \\
&\lesssim \sum_{Q \in \S}
\prod_{j=1}^m \|f_j v_j^{\frac1{p_j}}\|_{\bar{B_j}, Q}^p
\|v_j^{-\frac1{p_j}}\|_{B_j, Q}^p
\|u^{\frac1p}\|_{p, Q}^p  |Q|
\nonumber \\
&\lesssim ||(u, \vec{v})||_{A,\, \vec{B},\, \vec{p}}^p
\sum_{Q \in \S}\prod_{j=1}^m
\left(\inf_{Q} M_{\bar{B_j}}(f_j v_j^{\frac{1}{p_j}})\right)^p |E_Q|
\nonumber \\
&\le||(u, \vec{v})||_{A,\, \vec{B},\, \vec{p}}^p \prod_{j=1}^m \biggl(
\int_{\Rn} M_{\bar{B_j}}(f_j v_j^{\frac{1}{p_j}})(x)^{p_j} dx\biggr)^{p/{p_j}}
\nonumber \\
&\le ||(u, \vec{v})||_{A,\, \vec{B},\, \vec{p}}^p
\prod_{j=1}^m \|M_{\bar{B_j}}\|_{L^{p_j}(\Rn)}^{p}\|f_j\|_{L^{p_j}(v_j)}^{p},
\end{align}
where Lemma \ref{lem:MBp} is used in the last step.

Next let us deal with the case $2<p<\infty$. By duality, one has
\begin{align}\label{eq:AS-dual}
\|\A_{\S}^2(\vec{f})\|_{L^p(u)}^2 = \|\A_{\S}^2(\vec{f})^2\|_{L^{p/2}(u)}
=\sup_{\substack{0 \le h \in L^{(p/2)'}(u) \\ \|h\|_{L^{(p/2)'}(u)=1}}} \int_{\Rn} \A_{\S}^2(\vec{f})(x)^2 h(x) u(x) dx.
\end{align}
Fix a nonnegative function $h \in L^{(p/2)'}(u)$ with $\|h\|_{L^{(p/2)'}(u)}=1$. Then using H\"{o}lder's inequality \eqref{eq:Holder-AA}  and Lemma \ref{lem:MBp}, we obtain
\begin{align}\label{eq:ASp2}
&\int_{\Rn} \A_{\S}^2(\vec{f})(x)^2 h(x) u(x) dx
\lesssim \sum_{Q \in \S} \prod_{j=1}^m \langle |f_j| \rangle_{Q}^2 \langle hu \rangle_{Q} |Q|
\nonumber \\
&\lesssim \sum_{Q \in \S} \prod_{j=1}^m \|f_j v_j^{\frac{1}{p_j}}\|_{\bar{B_j}, Q}^2 \|v^{-\frac{1}{p_j}}\|_{B_j, Q}^2
\|hu^{1-\frac{2}{p}}\|_{\bar{A}, Q} \|u^{\frac{2}{p}}\|_{A, Q} |Q|
\nonumber \\
&\lesssim ||(u, \vec{v})||_{A,\, \vec{B},\, \vec{p}}^2 \sum_{Q \in \S}  \prod_{j=1}^m \left(\inf_{Q} M_{\bar{B_j}}(f_j v_j^{\frac{1}{p_j}})\right)^2
\left(\inf_{Q} M_{\bar{A}}(hu^{1-\frac{2}{p}})\right) |E_Q|
\nonumber \\
&\le ||(u, \vec{v})||_{A,\, \vec{B},\, \vec{p}}^2 \int_{\Rn} \prod_{j=1}^m  M_{\bar{B_j}}(f_j v_j^{\frac{1}{p_j}})(x)^2 M_{\bar{A}}(hu^{1-\frac{2}{p}})(x) dx
\nonumber \\
&\le ||(u, \vec{v})||_{A,\, \vec{B},\, \vec{p}}^2 \|\prod_{j=1}^m M_{\bar{B_j}}(f_j v_j^{\frac{1}{p_j}})^2\|_{L^{p/2}(\Rn)}
\|M_{\bar{A}}(hu^{1-\frac{2}{p}})\|_{L^{(p/2)'}(\Rn)}
\nonumber \\
&\le ||(u, \vec{v})||_{A,\, \vec{B},\, \vec{p}}^2 \prod_{j=1}^m\| M_{\bar{B_j}}(f_j v_j^{\frac{1}{p_j}})\|^2_{L^{p_j}(\Rn)}
\|M_{\bar{A}}(hu^{1-\frac{2}{p}})\|_{L^{(p/2)'}(\Rn)}
\nonumber \\
&\le ||(u, \vec{v})||_{A,\, \vec{B},\, \vec{p}}^2 \prod_{j=1}^m
\|M_{\bar{B_j}}\|_{\mathcal{L}(L^{p_j}(\Rn))}^2
\|f_j\|_{L^{p_j}(v_j)}^2 \|M_{\bar{A}}\|_{\mathcal{L}(L^{(p/2)'}(\Rn))} \|h\|_{L^{(p/2)'}(u)},
\end{align}
where $$\|M_{\bar{B_j}}\|_{\mathcal{L}(L^{p_j}(\Rn))}= \|M_{\bar{B_j}}\|_{L^{p_j}(\Rn)\to L^{p_j}(\Rn)}$$ and $$\|M_{\bar{A}}\|_{\mathcal{L}(L^{(p/2)'}(\Rn))}= \|M_{\bar{A}}\|_{L^{(p/2)'}(\Rn)\to L^{(p/2)'}(\Rn)}.$$
Therefore, \eqref{eq:ASLp} immediately follows from \eqref{eq:ASp1}, \eqref{eq:AS-dual} and \eqref{eq:ASp2}.
\end{proof}

\begin{proof}[\textbf{Proof of Theorem \ref{thm:FS}.}]
Fix exponents $\frac1p=\frac{1}{p_1}+\cdots+\frac{1}{p_m}$ with $1<p_1,\ldots, p_m<\infty$, and weights $\vec{w}=(w_1, \ldots, w_m)$. Note that $v_i(x):=Mw_i(x) \ge \langle w_i \rangle_{Q}$ for any dyadic cube $Q \in \S$ containing $x$. For each $i$, let $A_i$ be a Young function such that ${\bar{A}}_i \in B_{p_i}$.
By Lemma \ref{lem:MBp}, we have
\begin{align}\label{eq:MAf}
\|M_{\bar{A}_i} (f_i v_i^{\frac{1}{p_i}})\|_{L^{p_i}(\Rn)} \lesssim \|f_i\|_{L^{p_i}(v_i)},\quad i=1,\ldots,m.
\end{align}
Thus, using sparse domination \eqref{eq:S-sparse}, H\"{o}lder's inequality and \eqref{eq:MAf}, we deduce that
\begin{align*}
\|S_{\alpha}(\vec{f})\|_{L^p(\nu_{\vec{w}})}^p
&\lesssim \alpha^{pmn}\sum_{j=1}^{3^n} \sum_{Q \in \S_j}
\prod_{i=1}^m \langle |f_i| \rangle_{Q}^p \nu_{\vec{w}}(Q)
\\
&\le \alpha^{pmn}\sum_{j=1}^{3^n} \sum_{Q \in \S_j}
\prod_{i=1}^m \|f_i v_i^{\frac{1}{p_i}}\|_{\bar{A}_i, Q}^p
\|v_i^{-\frac{1}{p_i}}\|_{A_i, Q}^p \nu_{\vec{w}}(Q)
\\
&\le \alpha^{pmn}\sum_{j=1}^{3^n} \sum_{Q \in \S_j}
\prod_{i=1}^m \|f_i v_i^{\frac{1}{p_i}}\|_{\bar{A}_i, Q}^p
\langle w_i \rangle_Q^{-\frac{p}{p_i}} \langle \nu_{\vec{w}} \rangle_Q |Q|
\\
&\lesssim \alpha^{pmn}\sum_{j=1}^{3^n} \sum_{Q \in \S_j} \prod_{i=1}^m
\Big(\inf_{Q} M_{\bar{A}_i}(f_i v_i^{\frac{1}{p_i}}) \Big)^p |E_Q|
\\
&\lesssim\alpha^{pmn}\int_{\Rn}
\prod_{i=1}^m M_{\bar{A}_i}(f_i v_i^{\frac{1}{p_i}})(x)^p dx
\lesssim\alpha^{pmn}\prod_{i=1}^m \|M_{\bar{A}_i} (f_i v_i^{\frac{1}{p_i}})\|_{L^{p_i}(\Rn)}^p
\\
&\lesssim \alpha^{pmn}\prod_{i=1}^m \|f_i\|_{L^{p_i}(v_i)}^p
= \alpha^{pmn}\prod_{i=1}^m \|f_i\|_{L^{p_i}(Mw_i)}^p.
\end{align*}
This shows \eqref{eq:S-FS}. Likewise, one can obtain \eqref{eq:g-FS}.
\end{proof}

\section{Entropy bumps}\label{sec:entropy}
In this section, we will prove entropy bump inequalities (Theorem \ref{thm:entropy}). By the sparse domination for Littlewood-Paley operators, see \eqref{eq:S-sparse} and \eqref{eq:g-sparse}, it suffices to prove the results  for $\A_{\S}^{r}$, $r\geq 1$.

Let us call $(\alpha_i)=(\alpha_1,\alpha_2,\dots,\alpha_m)$. We will denote $(\alpha_{i})_{i\neq j}=(\alpha_1,\dots,\alpha_{j-1},\alpha_{j+1},\dots, \alpha_m)$. Having that notation at our disposal we define the following sub-multilinear maximal function.
\[
\M^{(\alpha_{i})_{i\neq j}}(\vec{\sigma})(x)
:=\sup_{x\in Q}\prod_{i\in\{1,\dots,m\},i\not=j}\langle\sigma_{i}\rangle_{Q}^{\alpha_{i}}
\]
and given $\vec{p}=(p_{1},\dots,p_{m})$
\[
\mathcal{M}^{\frac{1}{\vec{p}}}(\vec{\sigma})(x)
:=\sup_{x\in Q}\prod_{i=1}^{m}\langle\sigma_{i}\rangle_{Q}^{\frac{1}{p_{i}}}
\]
Let $1<p_{1},\dots,p_{m}<\infty$ and $\frac{1}{p}=\frac{1}{p_{1}}+\dots+\frac{1}{p_{m}}$.
We define
\[
\rho_{\vec{\sigma},\vec{p}}(Q)=
\bigg(\int_{Q}\mathcal{M}^{\frac{p}{\vec{p}}}(\sigma_{i}\chi_{Q})(x)dx\bigg)
\bigg(\int_Q \prod_{i=1}^{m}\sigma_i(x)^{\frac{p}{p_i}}dx \bigg)^{-1}.
\]
In the scalar case we shall denote just
\[
\rho_{\nu}(Q)=\frac{1}{\nu(Q)}\int_{Q}M(\nu\chi_{Q})(x)dx.
\]
Given an increasing function $\varepsilon:[1,+\infty)\rightarrow(0,+\infty)$
let us denote
\begin{equation*}
\rho_{\vec{\sigma},\vec{p},\varepsilon}(Q) :=\rho_{\vec{\sigma},\vec{p}}(Q)\varepsilon(\rho_{\vec{\sigma},\vec{p}}(Q))
\quad\text{and}\quad \rho_{\nu,\varepsilon}(Q) :=\rho_{\nu}(Q)\varepsilon(\rho_{\nu}(Q)).
\end{equation*}

With the notation we have just fixed, we are in the position to introduce the entropy bump conditions. For weights $\vec{\sigma}=(\sigma_1,\ldots,\sigma_m)$ and $\nu$, we define
\begin{equation}\label{eq:EntBump-1-1}
\lfloor \vec{\sigma},\nu\rfloor _{\vec{p},r,\varepsilon}
=\sup_{Q}\bigg(\prod_{i=1}^{m}\langle\sigma_{i}\rangle_{Q}^{\frac{p}{p'_i}}\bigg)
\langle\nu\rangle_{Q}\rho_{\vec{\sigma},\vec{p},\varepsilon}(Q) \rho_{\nu,\varepsilon}  (Q)^{\frac{p}{r}-1}.
\end{equation}
Also, if $\vec{\sigma}=(\sigma_1,\dots,\sigma_{m+1})$, we denote
\[
\left\lfloor \vec{\sigma}\right\rfloor _{\vec{q},\vec{p},\rho,\theta,j}
:=\sup_{Q}\prod_{i=1}^{m+1} \langle\sigma_{i}\rangle_{Q'}^{q_i} \left(\frac{\int_{Q}\M^{(1/(\theta p_{i}))_{i\neq j}}(\vec{\sigma})}{\int_{Q}\prod_{i\not=j}\sigma_{i}^{1/(\theta p_{i})}}\right)^{\theta}
\rho\left(\left(\frac{\int_{Q}\M^{(1/(\theta p_{i}))_{i\neq j}}(\vec{\sigma})}{\int_{Q}\prod_{i\not=j}\sigma_{i}^{1/(\theta p_{i})}}\right)^{\theta}\right).
\]
Denote $\overrightarrow{f\sigma} :=(f_{1}\sigma_{1},\dots,f_{m}\sigma_{m})$. Armed with the notation and the definitions of the entropy bumps just introduced, we can finally state and prove the main theorems of this section.

\begin{theorem}\label{thm:Convex}
Let $\frac{1}{p}=\frac{1}{p_1}+\cdots+\frac{1}{p_m}$ with $p>r$ and $1<p_1,\dots,p_m<\infty$. Let $\sigma_{1},\dots,\sigma_{m}$ and $\nu$ be weights. Assume that $\varepsilon$ is a monotonic increasing function on $(1,\infty)$ satisfying $ \int_{1}^{\infty}\frac{dt}{\varepsilon(t)t}<\infty$. Then
\begin{align}\label{eq:ASp>r}
\|\A_{\S}^{r}(\overrightarrow{f\sigma})\|_{L^{p}(\nu)}
\lesssim\left\lfloor \vec{\sigma},\nu\right\rfloor _{\vec{p},r,\varepsilon}^{\frac{1}{p}}\prod_{i=1}^{m}\|f\|_{L^{p_{i}}(\sigma_{i})}.
\end{align}
\end{theorem}

\begin{theorem}\label{thm:Concave}
Let $\frac{1}{p}=\frac{1}{p_1}+\cdots+\frac{1}{p_m}$ with $p\leq r$ and $1<p_1,\dots,p_m<\infty$. Let $\sigma_{1},\dots,\sigma_{m}$ and $\nu$ be weights. Assume that $\rho$ is a monotonic increasing function on $(1,\infty)$ satisfying $\int_{1}^{\infty}\frac{dt}{\rho(t)t}<\infty$.
Then
\begin{align}\label{eq:ASp<r}
\|\A_{\S}^{r}(\overrightarrow{f\sigma})\|_{L^{p}(\nu)}
\lesssim \lfloor \vec{\sigma},\nu \rfloor_{\vec{\frac{r}{p'}},\vec{p},\rho,\frac{r}{p},m+1}^{\frac{1}{r}}
\prod_{i=1}^{m}\|f\|_{L^{p_{i}}(\sigma_{i})},
\end{align}
where $\vec{p}=(p_1, \dots, p_m, p')$ and $\frac{r}{\vec{p}'}=\big(\frac{r}{p'_1},\ldots,\frac{r}{p'_m},\frac{r}{p}\big)$.
\end{theorem}


\subsection{Proof of Theorem \ref{thm:Convex}}
We need a multilinear version of Carleson embedding theorem from \cite{CD}.
\begin{lemma}\label{lem:CarEmb}
Let $\vec{\sigma}=(\sigma_1, \ldots, \sigma_m)$ be weights. Let  $1<p_{i}<\infty$ and $p\in(1,\infty)$ satisfying $\frac{1}{p}=\frac{1}{p_{1}}+\dots+\frac{1}{p_{m}}$. Assume that  $\{a_Q\}_{Q \in \D}$ is a sequence of non-negative numbers for which the following condition holds
\begin{equation}\label{eq:CarlCond-1}
\sum_{Q'\subset Q}a_{Q'} \leq A\int_{Q}\prod_{i=1}^{m}\sigma_{i}^{\frac{p}{p_{i}}}dx,\quad \forall Q\in\D.
\end{equation}
Then for all $f_{i}\in L^{p_{i}}(\sigma_{i})$,
\begin{align}\label{Carleson-1}
\bigg(\sum_{Q \in \D}a_{Q}\Big(\prod_{i=1}^{m}\fint_{Q} f_i\, d\sigma_i\Big)^{p}\bigg)^{\frac1p}
\leq A\prod_{i=1}^m p'_i \|f_i\|_{L^{p_i}(\sigma_i)}.
\end{align}
\end{lemma}

With this result in hand, we are in the position to settle Theorem \ref{thm:Convex} following ideas in \cite{LS}.
\begin{proof}[Proof of Theorem \ref{thm:Convex}]
First we split the sparse family as follows. We say that $Q\in\mathcal{S}_a$ if and only if
\[\bigg(\prod_{i=1}^{m}\langle\sigma_{i}\rangle_{Q}^{\frac{r}{p'_{i}}}\bigg)\langle w\rangle_{Q}^{\frac{r}{p}}\rho_{\vec{\sigma},\vec{p},\varepsilon}(Q)^{\frac{r}{p}}\rho_{\nu,\varepsilon}(Q)^{1/(p/r)'}\simeq 2^a\]
Let us begin providing a suitable estimate for each of those pieces of the sparse family. Given a weight $\gamma$ let us denote $\langle h\rangle_Q^\gamma:=\frac{1}{\gamma(Q)}\int_Q|h(x)|\gamma(x)dx$. Assume that $g\in L^{(p/r)'}(\nu)$. By duality we can write
\begin{align*}
& \bigg\langle \sum_{Q\in \S_a} \Big(\prod_{i=1}^{m}\langle f_{i}\sigma_{i}\rangle_{Q}\Big)^{r} {\bf 1}_{Q},g\nu\bigg\rangle
\\
& =\sum_{Q\in\S_a} \Big(\prod_{i=1}^{m}\langle f_i\rangle_{Q}^{\sigma_{i}}\Big)^{r} \Big(\prod_{i=1}^{m}\langle\sigma_{i}\rangle_{Q}\Big)^r \langle g\rangle_{Q}^{\nu}\langle\nu\rangle_{Q}\lvert Q\rvert
\\
&=\sum_{Q\in\S_a} \Big(\prod_{i=1}^{m}\langle f_i \rangle_{Q}^{\sigma_i} \Big)^r
\Big(\prod_{i=1}^{m}\langle\sigma_i \rangle_{Q}^{\frac{r}{p_i}}\Big)
\Big\{\Big(\prod_{i=1}^{m}\langle\sigma_i \rangle_{Q}^{\frac{r}{p'_i}} \Big)
\langle w \rangle_{Q}^{\frac{r}{p}}\Big\} \langle\nu\rangle_{Q}^{1/(p/r)'}\langle g\rangle_{Q}^{\nu}\cdot\lvert Q\rvert
\\
&=\sum_{Q\in\S_a} \Big(\prod_{i=1}^{m}\langle f_i \rangle_{Q}^{\sigma_i}\Big)^r
\frac{\prod_{i=1}^{m}\langle\sigma_{i}\rangle_{Q}^{\frac{r}{p_{i}}}}{\rho_{\vec{\sigma},\vec{p},\varepsilon}(Q)^{\frac{r}{p}}} \frac{\langle\nu\rangle_{Q}^{1/(p/r)'}}{\rho_{\nu,\varepsilon}(Q)^{1/(p/r)'}}\langle g\rangle_{Q}^{\nu}\cdot\lvert Q\rvert
\\
&\qquad\quad\times \Bigl\{\bigg(\prod_{i=1}^{m}\langle\sigma_{i}\rangle_{Q}^{\frac{r}{p'_{i}}}\bigg)\langle w\rangle_{Q}^{\frac{r}{p}}\rho_{\vec{\sigma},\vec{p},\varepsilon}(Q)^{\frac{r}{p}}\rho_{\nu,\varepsilon}(Q)^{1/(p/r)'}\Bigr\}
\\
& \lesssim2^a \sum_{Q\in\S_a} \Big(\prod_{i=1}^{m}\langle f_{i}\rangle_{Q}^{\sigma_{i}}\Big)^r \frac{\prod_{i=1}^{m}\sigma_{i}(Q)^{\frac{r}{p_{i}}}}{\rho_{\vec{\sigma},\vec{p},\varepsilon}(Q)}\frac{\nu(Q)^{1/(p/r)'}}{\rho_{\nu,\varepsilon}(Q)^{1/(p/r)'}}\langle g\rangle_{Q}^{\nu}
\\
& \leq 2^a \bigg(\sum_{Q\in\S_a} \Big(\prod_{i=1}^{m}\langle f_{i}\rangle_{Q}^{\sigma_{i}}\Big)^{p}\frac{\prod_{i=1}^{m}\sigma_{i}(Q)^{\frac{p}{p_{i}}}}{\rho_{\vec{\sigma},\vec{p},\varepsilon}(Q)}\bigg)^{\frac{r}{p}}\bigg(\sum_{Q\in\S_a} (\langle g\rangle_{Q}^{\nu})^{(p/r)'}\frac{\nu(Q)}{\rho_{\nu,\varepsilon}(Q)}\bigg)^{\frac{1}{(p/r)'}}.
\end{align*}
For the second term, we would like to get that
\begin{align}\label{eq:QSa-1}
\sum_{Q\in \S_a}(\langle g\rangle_{Q}^{\nu})^{(p/r)'}\frac{\nu(Q)}{\rho_{\nu,\varepsilon}(Q)}
\lesssim \|g\|_{L^{(p/r)'}(\nu)}^{(p/r)'}.
\end{align}
We omit the proof of \eqref{eq:QSa-1} and focus on the first term above, since the argument that we are going provide, essentially contains the linear case. For the first term, it needs to show
\begin{align}\label{eq:QSa-2}
\sum_{Q\in\S_a} \Big(\prod_{i=1}^{m}\langle f_{i}\rangle_{Q}^{\sigma_{i}}\Big)^{p}
\frac{\prod_{i=1}^{m}\sigma_{i}(Q)^{\frac{p}{p_{i}}}}{\rho_{\vec{\sigma},\vec{p},\varepsilon}(Q)}
\lesssim\prod_{i=1}^{m}\|f_i\|_{L^{p_i}(\sigma_i)}^p.
\end{align}
Taking into account Lemma \ref{lem:CarEmb}, it suffices to verify that \eqref{eq:CarlCond-1} holds with
\[
a_{Q}=\begin{cases}
\frac{\prod_{i=1}^{m}\sigma_{i}(Q)^{\frac{p}{p_{i}}}}{\rho_{\vec{\sigma},\vec{p},\varepsilon}(Q)} & Q\in\S_a,\\
0 & \text{otherwise}.
\end{cases}
\]
Indeed, let us call $\S_a(R)$ the set of cubes of $\S_a$ that are contained in $R\in\D$. Then
\begin{align*}
\sum_{Q\in\S_a(R)} &\frac{\prod_{i=1}^{m}\sigma_{i}(Q)^{\frac{p}{p_{i}}}}{\rho_{\vec{\sigma},\vec{p},\varepsilon}(Q)}
\lesssim\sum_{j=1}^{\infty}\sum_{\substack{\rho_{\vec{\sigma},\vec{p}}(Q)\sim2^{j} \\ Q\in\S_a(R)}}
\frac{\prod_{i=1}^{m}\sigma_{i}(Q)^{\frac{p}{p_{i}}}}{\rho_{\vec{\sigma},\vec{p},\varepsilon}(Q)}
\\
& \lesssim \sum_{j=1}^{\infty}\sum_{{\substack{\text{maximal }Q\in\S_a(R)\\ \rho_{\vec{\sigma},\vec{p}}(Q)\simeq2^j}}} \sum_{\substack{P \subset Q \\ P \in\S_a(R)}} \frac{\prod_{i=1}^{m}\sigma_{i}(P)^{\frac{p}{p_{i}}}}{\rho_{\vec{\sigma},\vec{p},\varepsilon}(P)}
\\
& \leq\sum_{j=1}^{\infty}\sum_{\substack{\text{maximal } Q \in\S_a(R) \\ \rho_{\vec{\sigma},\vec{p}}(Q)\simeq2^j}}\sum_{\substack{P\subset Q\\ P \in \S_a(R)}} \frac{2^{-j}}{\varepsilon(2^j)} \int_{E_P} \M^{\frac{p}{\vec{p}}}(\sigma_{i}{\bf 1}_{Q})(x) dx
\\
& \lesssim\sum_{j=1}^{\infty}\sum_{\substack{\text{maximal } Q \in\S_a(R)\\ \rho_{\vec{\sigma},\vec{p}}(Q)\simeq2^j}} \frac{2^{-j}}{\varepsilon(2^j)} \int_{Q}\M^{\frac{p}{\vec{p}}}(\sigma_{i}{\bf 1}_{Q})(x)  dx
\\
&\lesssim\Big(\prod_{i=1}^{m}\sigma_{i}^{\frac{p}{p_{i}}}\Big)(R) \sum_{j=0}^{\infty}\frac{1}{\varepsilon(2^{j})}
\lesssim\Big(\prod_{i=1}^{m}\sigma_{i}^{\frac{p}{p_{i}}}\Big)(R) \int_{1}^{\infty}\frac{dt}{t\varepsilon(t)}.
\end{align*}
This provides the desired bound.

Collecting \eqref{eq:QSa-1} and \eqref{eq:QSa-2}, we have shown that
\[
\bigg\langle \sum_{Q\in\S_a}\Big(\prod_{i=1}^{m}\langle f_{i}\sigma_{i}\rangle_{Q}\Big)^r {\bf 1}_{Q},g\nu\bigg\rangle \lesssim 2^a \prod_{i=1}^{m}\|f_i\|_{L^{p_i}(\sigma_i)}^r \cdot \|g\|_{L^{(p/r)'}(\nu)}.
\]
Since for the largest $a$ for which $\S_a$ is not empty we have that $\lfloor \sigma,w\rfloor _{p,r,\varepsilon}^{\frac{r}{p}}\simeq2^a$,  summing in $a$ yields
\[
\bigg\langle \sum_{Q\in\S} \Big(\prod_{i=1}^{m}\langle f_{i}\sigma_{i}\rangle_{Q}\Big)^r {\bf 1}_Q, g\nu\bigg\rangle \lesssim \lfloor \vec{\sigma},\nu \rfloor _{\vec{p}, r, \varepsilon}^{\frac{r}{p}} \prod_{i=1}^{m}\|f_i\|_{L^p(\sigma_i)}^r \|g\|_{L^{(p/r)'}(w)}.
\]
Consequently,
\[
\|\A_{\S}^r(\overrightarrow{f\sigma})\|_{L^p(\nu)}
\lesssim \lfloor \vec{\sigma},\nu\rfloor _{\vec{p}, r, \varepsilon}^{\frac{1}{p}}
\prod_{i=1}^{m}\|f_{i}\|_{L^{p}(\sigma_{i})}.
\]
This shows Theorem \ref{thm:Convex}.
\end{proof}

\subsection{Proof of Theorem \ref{thm:Concave}}

To settle Theorem \ref{thm:Concave} we are going to follow the scheme in \cite{ZK}. First we borrow a result from \cite{COV}.
\begin{lemma}
\label{lem:cov} For every $1<s<\infty$ we have that for every positive locally finite measure $\sigma$ on $\Rn$ and any positive numbers
$\lambda_{Q}$, $Q\in\mathcal{D}$, we have
\[
\int_{\Rn} \Big(\sum_{Q\in\D}\frac{\lambda_{Q}}{\sigma(Q)}{\bf 1}_{Q}(x)\Big)^{s}d\sigma(x)
\lesssim_{s}\sum_{Q\in\mathcal{D}}\lambda_{Q}\Big(\sigma(Q)^{-1}\sum_{Q'\subseteq Q}\lambda_{Q'}\Big)^{s-1}.
\]
\end{lemma}

Given a sparse family $\S$ contained in a dyadic grid $\D$, for every $Q\in\S$ we will denote $\S(Q)$ the family of cubes of $\S$ that are contained in $Q$.

\begin{lemma}\label{lem:sum<1}
Let $\beta_{1},\dots,\beta_{m}\ge 0$ be such that $\beta:=\sum_{i=1}^{m}\beta_{i}<1$. Let $\S \subset \D$ be a sparse family.  Then for every cube $Q \in \S$ and all functions $w_{1},\dots,w_{m}$,
\[
\sum_{Q' \in\S(Q)} |Q'| \prod_{i=1}^{m} \langle w_i\rangle_{Q'}^{\beta_i}
\lesssim |Q| \prod_{i=1}^{m}  \langle w_i\rangle_{Q}^{\beta_i} .
\]
\end{lemma}



\begin{lemma}
\label{lem:key} Let $j\in\{1,\dots,m\}$, $s_{1},\dots,s_{m}\in\mathbb{R}$ with $s_{i}>0$
for each $i\in\{1,\dots,m\}$ with $i\neq j$, and $q_{1},\dots,q_{m}>0$ with $q_{j}=1+s_{j}$ be
such that
\[
\sum_{i}s_{i}\leq\sum_{i}q_{i},\quad\frac{\sum_{i}s_{i}}{\sum_{i}q_{i}}<\min_{i\neq j}\frac{s_{i}}{q_{i}}.
\]
Let $\mathcal{S}$ be a sparse family such that for every $Q\in\mathcal{S}$
\begin{equation}
2^{r}\leq\bigg(\frac{\int_{Q}\mathcal{M}^{(1/(\theta p_{i}))_{i\neq j}}(\vec{w})}{\int_{Q}\prod_{i\not=j}w_{i}^{1/(\theta p_{i})}}\bigg)^{\theta}\leq2^{r+1}\label{eq:PropS}
\end{equation}
where $p_{i}\in(0,+\infty)$. Then for every $0<\alpha<\infty$ we
have that
\[
\mathscr{A} :=\int_{\Rn} \Big(\sum_{Q\in\mathcal{S}}\prod_{i=1}^{m} \langle w_i \rangle_{Q}^{s_{i}\alpha}{\bf 1}_{Q}\Big)^{\frac{1}{\alpha}}dw_{j}\lesssim_{\vec{s},\vec{q},\alpha}\frac{\left\lfloor \vec{w}\right\rfloor _{\vec{q},\vec{p},\rho,\theta,j}}{2^{r}\rho(2^{r})}\sum_{Q\in\mathcal{S}}|Q|\prod_{i\neq j} \langle w_i \rangle_{Q}^{s_{i}-q_{i}}.
\]
\end{lemma}

\begin{proof}
Then the left-hand side of the conclusion is monotonically decreasing in $\alpha$ and the right-hand side does not depend on $\alpha$, so it suffices to consider small $\alpha$, in particular we may assume $s_{i}\alpha+\delta_{ij}>0$ and $\alpha<1$.

It follows from the hypothesis that for sufficiently small $\alpha$ there exists an $\epsilon$ such that
\[
\frac{\alpha\sum_{i}s_{i}}{\sum_{i}q_{i}}<\epsilon
\leq \min\bigg\{\frac{1}{1/\alpha-1},\, \min_{i}\frac{\alpha s_{i}+\delta_{ij}}{q_{i}}\bigg\}.
\]
By the assumption $\alpha<1$ and Lemma~\ref{lem:cov},
\[
\mathscr{A} \simeq\sum_{Q\in\mathcal{S}}|Q|\prod_{i=1}^{m} \langle w_i \rangle_{Q}^{\alpha s_{i}+\delta_{ij}}\Big(\sum_{Q'\in\mathcal{S}(Q)} \frac{|Q'|}{w_{j}(Q)}\prod_{i=1}^{m} \langle w_i \rangle_{Q'}^{\alpha s_{i}+\delta_{ij}}\Big)^{\frac{1}{\alpha}-1}.
\]
Taking into account the definition of $\left\lfloor \vec{w}\right\rfloor _{\vec{q},\rho,\theta,j}$ and \eqref{eq:PropS}, we get
\[
\mathscr{A} \lesssim
\bigg(\frac{\left\lfloor \vec{w}\right\rfloor _{\vec{q},\vec{p},\rho,\theta,j}}{2^{r}\rho(2^{r})}\bigg)^{\epsilon(\frac{1}{\alpha}-1)}\sum_{Q\in\mathcal{S}}|Q|\prod_{i=1}^{m} \langle w_i \rangle_{Q}^{\alpha s_{i}+\delta_{ij}}\Big(\sum_{Q'\in\mathcal{S}(Q)} \frac{|Q'|}{w_{j}(Q)}\prod_{i=1}^{m} \langle w_i \rangle_{Q'}^{\alpha s_{i}+\delta_{ij}-\epsilon q_{i}}\Big)^{\frac{1}{\alpha}-1}.
\]
Observe that $\alpha s_{i}+\delta_{ij}-\epsilon q_{i}\geq0$ and $\sum_{i}(\alpha s_{i}+\delta_{ij}-\epsilon q_{i})<1$. Hence, Lemma~\ref{lem:sum<1} implies that
\begin{align*}
\mathscr{A} & \lesssim\bigg(\frac{\left\lfloor \vec{w}\right\rfloor _{\vec{q},\vec{p},\rho,\theta,j}}{2^{r}\rho(2^{r})}\bigg)^{\epsilon(\frac{1}{\alpha}-1)}\sum_{Q\in\mathcal{S}}|Q|\prod_{i=1}^{m} \langle w_i \rangle_{Q}^{\alpha s_{i}+\delta_{ij}}\Big(\frac{|Q|}{w_{j}(Q)}\prod_{i=1}^{m} \langle w_i \rangle_{Q}^{\alpha s_{i}+\delta_{ij}-\epsilon q_{i}}\Big)^{\frac{1}{\alpha}-1}
\\
& =\left(\frac{\left\lfloor \vec{w}\right\rfloor _{\vec{q},\vec{p},\rho,\theta,j}}{2^{r}\rho(2^{r})}\right)^{\epsilon(\frac{1}{\alpha}-1)}\sum_{Q\in\mathcal{S}}|Q|\prod_{i=1}^{m} \langle w_i \rangle_{Q}^{\delta_{ij}+s_{i}-\epsilon q_{i}(\frac{1}{\alpha}-1)}.
\end{align*}
By construction $1-\epsilon(1/\alpha-1)\geq0$, and again by the definition of $\left\lfloor \vec{w}\right\rfloor _{\vec{q},\rho,\theta,j}$ and \eqref{eq:PropS}, we conclude that
\[
\mathscr{A} \leq \frac{\left\lfloor \vec{w}\right\rfloor _{\vec{q},\vec{p},\rho,\theta,j}}{2^{r}\rho(2^{r})}
\sum_{Q}|Q|\prod_{i=1}^{m} \langle w_i \rangle_{Q}^{\delta_{ij}+s_{i}-q_{i}}.
\]
and we are done.
\end{proof}

Now we present a stopping time condition. Let $\mathcal{S}\subset\mathcal{D}$
be a finite sparse family let $\lambda_{i}:\mathcal{S}\to[0,\infty)$,
$Q\mapsto\lambda_{i,Q}$ be a function that takes a cube to a non-negative
real number. Then we have that $\mathcal{F}_{i}$ is the minimal family
of cubes such that the maximal members of $\mathcal{S}$ are contained
in $\mathcal{F}_{i}$, and if $F\in\mathcal{F}_{i}$, then every maximal
subcube $F'\subset F$ with $\lambda_{i,F'}\geq2\lambda_{i,F}$ is
also a member of $\mathcal{F}_{i}$.

For each cube $Q$ let $\pi_{i}(Q)$ (the parent of $Q$ in the stopping
family $\mathcal{F}_{i}$) be the smallest cube with $Q\subseteq\pi_{i}(Q)\in\mathcal{F}_{i}$.
We write $\sum_{F_{1},\dots,F_{m}}$ for the sum running over $F_{i}\in\mathcal{F}_{i}$.
We also write
\[
M\lambda_{i}(x):=\sup_{x\in Q\in\mathcal{D}}\lambda_{i,Q}.
\]

\begin{lemma}\label{lem:char}
Let $m\geq2$, $0<p_{1},\dots,p_{m-1}<\infty$. Define $\alpha:=\sum_{i=1}^{m-1}1/p_{i}$ and
\[
0<q_{i}:=s_{i}-\begin{cases}
1/p_{i}, & i<m,\\
1-\alpha, & i=m.
\end{cases}
\]
Assume that $\mathcal{S}$ is a sparse family such that for every $Q\in\S$,
\begin{equation}\label{eq:CondSparse}
2^{r}\leq\left(\frac{\int_{Q}\mathcal{M}^{(1/(\alpha p_{i}))_{i\neq m}}(\vec{w})}{\int_{Q}\prod_{i=1}^{m-1}w_{i}^{1/(\alpha p_{i})}}\right)^{\alpha}\leq2^{r+1}.
\end{equation}
Then one has
\begin{align*}
\mathscr{B} \lesssim [\vec{w}]_{\vec{q},\vec{p},\rho,\alpha,m} \rho(2^r)^{-1} \prod_{i=1}^{m-1}\|M\lambda_i\|_{L^{p_i}(w_i)},
\end{align*}
where
\begin{align*}
\mathscr{B} := \bigg(\sum_{F_{1},\dots,F_{m-1}}\prod_{i=1}^{m-1}\lambda_{i,F_{i}}^{\frac{1}{\alpha}}
\int \Big(\sum_{Q:\pi_{i}(Q)=F_{i}}\sum {\bf 1}_{Q}\prod_{i=1}^{m} \langle w_i \rangle_{Q}^{s_{i}-\delta_{im}}\Big)^{\frac{1}{\alpha}}dw_{m}\bigg)^{\alpha}.
\end{align*}
\end{lemma}

\begin{proof}
We will estimate $\mathscr{B}$ by means of Lemma \ref{lem:key} with $\tilde{s}_i=(s_i-\delta_{im})/\alpha$, $i\leq m$, and $\tilde{q}_{i}=q_{i}/\alpha$.
We can provide such an estimate since
\[
\sum_{i\leq m}\alpha\tilde{q}_{i}=\sum_{i\leq m}s_{i}-(1-\alpha)-\sum_{i<m}1/p_{i}=\sum_{i\leq m}s_{i}-1=\sum_{i\leq m}\alpha\tilde{s}_{i}.
\]
This yields that the first inequality in the hypothesis of the lemma holds, and for $i<m$ we have $\tilde{q}_i<\tilde{s}_i$, verifying the second inequality. Then, there holds
\begin{align}\label{eq:firstStep}
\mathscr{B} \lesssim \Big(\sum_{F_{1},\dots,F_{m-1}}\prod_{i=1}^{m-1}\lambda_{i,F_{i}}^{1/\alpha}\frac{[\vec{w}]_{\vec{q},\vec{p,}\rho,\alpha,m}}{2^{r}\rho(2^{r})}\sum_{Q:\pi_{i}(Q)=F_{i}}|Q|\prod_{i=1}^{m-1} \langle w_i \rangle_{Q}^{\tilde{s}_{i}-\tilde{q}_{i}}\Big)^{\alpha}.
\end{align}
The sparseness of $\mathcal{S}$ gives that
\begin{align*}
&\sum_{Q:\pi_i(Q)=F_i} |Q|\prod_{i=1}^{m-1}(w_{i})_{Q}^{\tilde{s}_{i}-\tilde{q}_{i}}
\lesssim\sum_{Q:\pi_{i}(Q)=F_{i}}|E_{Q}|\bigg(\prod_{i=1}^{m-1} \langle w_i \rangle_{Q}^{\frac{1}{p_i}}\bigg)^{\frac{1}{\alpha}}
\\
&\leq\sum_{Q:\pi_{i}(Q)=F_{i}}\int_{E_{Q}}\mathcal{M}^{(1/(\alpha p_{i}))_{i\neq m}}(\vec{w})
\leq \int_{F_{1}\cap\dotsb\cap F_{m-1}}\mathcal{M}^{(1/(\alpha p_{i}))_{i\neq m}}(\vec{w}).
\end{align*}
Thus, it follows from \eqref{eq:CondSparse} and H\"older inequality that
\begin{align*}
\mathscr{B} &\lesssim \frac{[\vec{w}]_{\vec{q},\vec{p,}\rho,\alpha,m}}{2^{j}\rho(2^{j})} \bigg(\sum_{F_{1},\dots,F_{m-1}}\prod_{i=1}^{m-1}\lambda_{i,F_{i}}^{\frac{1}{\alpha}} 2^{\frac{j}{\alpha}}\int_{F_{1}\cap\dots\cap F_{m-1}}\prod_{i=1}^{m-1}w_{i}^{\frac{1}{\alpha p_i}} \bigg)^{\alpha}
\\
& =\frac{[\vec{w}]_{\vec{q},\vec{p},\rho,\alpha,m}}{2^{j}\rho(2^{j})}2^{j}\bigg(\int\prod_{i=1}^{m-1}\sum_{F_{i}}1_{F_{i}}\lambda_{i,F_{i}}^{\frac{1}{\alpha}} w_{i}^{\frac{1}{\alpha p_i}} \bigg)^{\alpha}
\\
& \leq\frac{[\vec{w}]_{\vec{q},\vec{p},\rho,\alpha,m}}{\rho(2^{j})}\prod_{i=1}^{m-1}\bigg(\int\bigl(\sum_{F_{i}}1_{F_{i}}\lambda_{i,F_{i}}^{\frac{1}{\alpha}} \bigr)^{\alpha p_{i}}w_{i}\bigg)^{\frac{1}{p_i}}.
\end{align*}
We end the proof noticing that
\[
\bigg(\sum_{F_i}1_{F_i} \lambda_{i, F_i}^{\frac{1}{\alpha}}\bigg)^{\alpha} \simeq M\lambda_i,
\]
since at each point, the sum on the left-hand side is geometrically
increasing and, consequently, it is comparable to the last term.
\end{proof}

\begin{lemma}\label{lem:Concave}
Let $m\geq2$ and $0<p_{i},s_{i}<\infty$, $1\leq i<m$, and let $\alpha:=\sum_{i=1}^{m-1}1/p_{i}$. Suppose $q_{i}:=s_{i}-1/p_{i}>0$ for $i<m$ and let $q_{m}:=\alpha$. Then for every sparse family $\S$ and $\alpha\geq1$,
\begin{equation}\label{eq:Concave}
\bigg\|\sum_{Q\in\S} \prod_{i=1}^{m-1}\lambda_{i,Q} \langle w_i \rangle_{Q}^{s_{i}}{\bf 1}_{Q}\bigg\|_{L^{1/\alpha}(w_{m})}
\lesssim [\vec{w}]_{\vec{q},\vec{p},\rho,\alpha,m}\prod_{i=1}^{m-1}\|M\lambda_{i}\|_{L^{p_{i}}(w_{i})},
\end{equation}
provided that $\int_{1}^{\infty}\rho(t)\frac{dt}{t}<\infty$.
\end{lemma}

\begin{proof}
First we split $\mathcal{S}$ as follows
\[
\sum_{Q\in\mathcal{S}}\prod_{i=1}^{m-1}\lambda_{i,Q} \langle w_i \rangle_{Q}^{s_{i}} {\bf 1}_{Q}=\sum_{j}\sum_{Q\in\mathcal{S}_{j}}\prod_{i=1}^{m-1}\lambda_{i,Q} \langle w_i \rangle_{Q}^{s_{i}} {\bf 1}_{Q}
\]
where
\[
Q\in\mathcal{S}_{j} \iff 2^{j} \leq\bigg(\frac{\int_{Q}\mathcal{M}^{(1/(\alpha p_{i}))_{i\neq m}}(\vec{w})}{\int_{Q}\prod_{i=1}^{m-1}w_{i}^{1/(\alpha p_{i})}}\bigg)^{\alpha}\leq2^{j+1}.
\]
Then one has
\begin{equation}\label{eq:lamlam}
\bigg\|\sum_{Q\in S}\prod_{i=1}^{m-1}\lambda_{i,Q} \langle w_i \rangle_{Q}^{s_{i}} {\bf 1}_{Q}\bigg\|_{L^{1/\alpha}(w_{m})}
\lesssim \sum_{j} \bigg\|\sum_{Q\in\mathcal{S}_{j}}\prod_{i=1}^{m-1}\lambda_{i,Q} \langle w_i \rangle_{Q}^{s_{i}} {\bf 1}_{Q}\bigg\|_{L^{1/\alpha}(w_{m})}.
\end{equation}
 The right-hand side of \eqref{eq:lamlam} can be estimated by
\[
\bigg(\int\Big(\sum_{F_{1},\dots,F_{m-1}}\prod_{i=1}^{m-1}\lambda_{i,F_{i}}\sum_{Q:\pi_{i}(Q)=F_{i}}\prod_{i=1}^{m-1} \langle w_i \rangle_{Q}^{s_i} {\bf 1}_{Q}\Big)^{1/\alpha}dw_{m}\bigg)^{\alpha}.
\]
By subadditivity of the function $x\mapsto x^{1/\alpha}$, this is
bounded by
\[
\bigg(\sum_{F_{1},\dots,F_{m-1}}\prod_{i=1}^{m-1}\lambda_{i,F_{i}}^{1/\alpha}\int\Big(\sum_{Q:\pi_{i}(Q)=F_{i}}\prod_{i=1}^{m-1} \langle w_i \rangle_{Q}^{s_i} {\bf 1}_{Q}\Big)^{1/\alpha}dw_{m}\bigg)^{\alpha}.
\]
Therefore, Lemma~\ref{lem:char} applied to $s_m=1$ gives \eqref{eq:Concave} as desired.
\end{proof}

\begin{proof}[Proof of Theorem \ref{thm:Concave}]
We rewrite
\[
\|\A_{\S}^r(\overrightarrow{f\sigma})\|_{L^{p}(\nu)}
=\bigg\|\sum_{Q\in \S}\bigg(\prod_{i=1}^{m}\langle f_i \sigma_i \rangle_{Q}\bigg)^r1_{Q}\bigg\|_{L^{\frac{p}{r}}(\sigma_{m+1})}^{\frac1r}.
\]
Taking $m=m+1$, $w_{i}=\sigma_{i}$, $w_{m+1}=\nu$, $\lambda_{i,Q}=\left(\langle f_{i}\rangle_{Q}^{\sigma_{i}}\right)^{r}$,
$s_{i}=r$ , and $\alpha=\frac{r}{p}=\sum_{i=1}^{m}\frac{r}{p_i}$, we have $\tilde{q_{i}}:=r-r/p_{i}$ and by Lemma \ref{lem:Concave}
\begin{align*}
& \bigg\|\sum_{Q\in S}\bigg(\prod_{i=1}^{m}\langle f_i\sigma_i\rangle_{Q}\bigg)^r {\bf 1}_{Q}\bigg\|_{L^{\frac{p}{r}}(\nu)}
\lesssim \lfloor \vec{\sigma},\nu \rfloor_{\frac{r}{\vec{p}'},\vec{p},\rho,\frac{r}{p},m+1}\prod_{i=1}^{m}\|(M_{\sigma_{i}}f)^{r}\|_{L^{\frac{p_{i}}{r}}(\sigma_{i})}
\\
&\qquad =\lfloor \vec{\sigma},\nu \rfloor_{\frac{r}{\vec{p}'},\vec{p},\varepsilon,\frac{r}{p},m+1}\prod_{i=1}^{m}\|M_{\sigma_i}f\|_{L^{p_i}(\sigma_{i})}^{r}
\lesssim \lfloor\vec{\sigma},\nu \rfloor_{\frac{r}{\vec{p}'},\vec{p},\varepsilon,\frac{r}{p},m+1}\prod_{i=1}^{m}\|f\|_{L^{p_i}(\sigma_i)}^r.
\end{align*}
Hence,
\begin{equation*}
\|\A_{\S}^r(\overrightarrow{f\sigma})\|_{L^p(\nu)}
=\bigg\|\sum_{Q \in \S}\bigg(\prod_{i=1}^{m} \langle f_i \sigma_i \rangle_{Q} \bigg)^r {\bf 1}_Q \bigg\|_{L^{\frac{p}{r}}(\nu)}^{\frac{1}{r}}
\lesssim \lfloor \vec{\sigma},\nu \rfloor_{\frac{r}{\vec{p}'},\vec{p},\varepsilon,\frac{r}{p},m+1}^{\frac1r} \prod_{i=1}^{m}\|f\|_{L^{p_i}(\sigma_i)}
\end{equation*}
as we wanted to show.
\end{proof}

\section{Mixed weak type estimates}

The goal of this section is devoted to presenting the proof of Theorem \ref{thm:mixed}. To this end, we first establish a Coifman-Fefferman inequality with the precise $A_{\infty}$ weight constant.

\subsection{A Coifman-Fefferman inequality}

\begin{theorem}\label{thm:CF}
Let $\alpha \ge 1$. Then for every $0<p<\infty$ and for every $w\in A_{\infty}$,
\begin{align}\label{eq:CF}
\|S_{\alpha}(\vec{f})\|_{L^{p}(w)} &\lesssim \alpha^{mn} (p+1)
[w]_{A_{\infty}}^{\frac{1}{2}}\|\mathcal{M}(\vec{f})\|_{L^{p}(w)}.
\end{align}
\end{theorem}

\vspace{0.2cm}
\noindent$\bullet$ {\bf Sparse approach for $p\protect\geq2$.}
Considering \eqref{eq:S-sparse}, we are going to show that
\begin{align}\label{eq:ASCF}
\|\A_{\S}^{2}(\vec{f})\|_{L^{p}(w)}
\lesssim[w]_{A_{\infty}}^{\frac{1}{2}}\|\M(\vec{f})\|_{L^{p}(w)}, \quad\forall p \ge 2.
\end{align}
Without loss of generality, we shall assume that $f_{i}\geq0$, $i=1,\ldots,m$. Note that
\begin{align}\label{eq:AS-dual}
\|\A_{\S}^{2}(\vec{f})\|_{L^{p}(w)}^2
=\sup_{\substack{0 \le g \in L^{(p/2)'}(w) \\ \|g\|_{L^{(p/2)'}(w)=1} }}
\bigg|\sum_{Q \in \S} \prod_{i=1}^{m}\langle f_i\rangle_Q^2 \fint_Q g\, dw\, w(Q)\bigg|.
\end{align}
Fix $0 \le g \in L^{(p/2)'}(w)$ with $\|g\|_{L^{(p/2)'}(w)}=1$. We are going to split the sparse family in terms of principal cubes. Set
\[
\tau(P) :=\prod_{i=1}^{m}\langle f_i \rangle_P^2 \fint_{P}g \, dw,
\]
and consider $\mathcal{F}_{0}$ the family of maximal cubes of $\mathcal{S}$. We define
\[
\mathcal{F} :=\bigcup_{i=0}^{\infty}\mathcal{F}_{i} \quad\text{and}\quad
\mathcal{F}_i :=\bigcup_{Q\in\mathcal{F}_{i-1}}\left\{ P\subsetneq Q\text{ maximal}: \tau(P)>2\tau(Q)\right\} .
\]
For this family of cubes, we have that
\begin{align}\label{eq:Sfg}
&\sum_{Q\in\S} \prod_{i=1}^{m}\langle f_i\rangle_Q^2 \fint_{Q}g\, dw \, w(Q)
\nonumber\\
& \leq\sum_{P\in\mathcal{F}} \prod_{i=1}^{m}\langle f_i \rangle_P^2 \fint_{P}g\, dw \sum_{Q\in \S:\pi(Q)=P}w(Q)
\nonumber\\
& \lesssim[w]_{A_{\infty}} \sum_{P\in\mathcal{F}} \prod_{i=1}^{m}\langle f_i \rangle_P^2 \fint_{P}g\, dw\, w(P)
\nonumber\\
& \lesssim[w]_{A_{\infty}} \int_{\Rn} \M(\vec{f})(x)^2 M_{w}g(x)\, w(x)dx
\nonumber\\
& \lesssim[w]_{A_{\infty}} \|\M(\vec{f})^2\|_{L^{p/2}(w)} \|g\|_{L^{(p/2)'}(w)}.
\end{align}
Thus, \eqref{eq:AS-dual} and \eqref{eq:Sfg} immediately lead \eqref{eq:ASCF}.
\qed

\vspace{0.3cm}
\noindent$\bullet$ {\bf $M_{\delta}^{\sharp}$ approach.}
We next deal with the general case $0<p<\infty$. Recall that the sharp maximal function of $f$ is defined by
\[
M_{\delta}^{\sharp}(f)(x):=\sup_{x\in Q} \inf_{c \in \R} \bigg(\fint_{Q}|f^{\delta}-c|dx\bigg)^{\frac{1}{\delta}}.
\]
It was proved in \cite{OCPR} that for every $0<p<\infty$ and $\delta\in(0,1)$,
\begin{align}\label{eq:fMsharp}
\|f\|_{L^{p}(w)}\lesssim(p+1)[w]_{A_{\infty}}\|M_{\delta}^{\sharp}(f)\|_{L^{p}(w)}.
\end{align}

Let $\Phi$ be a fixed Schwartz function such that $\mathbf{1}_{B(0, 1)}(x) \le \Phi(x) \le \mathbf{1}_{B(0, 2)}(x)$. We define
\begin{align}\label{eq:S-def}
\widetilde{S}_{\alpha}(\vec{f})(x):=\bigg(\iint_{\R^{n+1}_{+}} \Phi\Big(\frac{x-y}{\alpha t}\Big) |\psi_t(\vec{f})(y)|^2\frac{dydt}{t^{n+1}}\bigg)^{1/2}.
\end{align}
It is easy to verify that
\begin{align}\label{eq:SSS}
S_{\alpha}(\vec{f})(x) \le \widetilde{S}_{\alpha}(\vec{f})(x) \le S_{2\alpha}(\vec{f})(x).
\end{align}
We note here that
\begin{align}\label{eq:SS-end}
\|\widetilde{S}_{\alpha}(\vec{f})\|_{L^{1/m,\infty}(\Rn)}\lesssim \alpha^{mn}\prod_{j=1}^m \|f_j\|_{L^1(\Rn)}.
\end{align}
In fact, by \cite[Lemma~3.1]{BH} and the endpoint estimate for $S_1$, we get
\begin{align*}
\|\widetilde{S}_{\alpha}(\vec{f})\|_{L^{1/m,\infty}(\Rn)}
&\le \|S_{2\alpha}(\vec{f})\|_{L^{1/m,\infty}(\Rn)}
\\
&\lesssim \alpha^{mn} \|S_{1}(\vec{f})\|_{L^{1/m,\infty}(\Rn)}
\lesssim \alpha^{mn} \prod_{j=1}^m \|f_j\|_{L^1(\Rn)}.
\end{align*}

Now, combining \eqref{eq:SSS}, \eqref{eq:fMsharp} and Lemma \ref{lem:Msharp} below, we conclude that
\begin{align*}
\|S_{\alpha}(\vec{f})\|_{L^{p}(w)}
& \leq\|\widetilde{S}_{\alpha}(\vec{f})\|_{L^{p}(w)}
\leq\|\widetilde{S}_{\alpha}(\vec{f})^{2}\|_{L^{p/2}(w)}^{\frac{1}{2}}\\
 & \leq(p+1)[w]_{A_{\infty}}^{\frac{1}{2}}\|M_{\gamma}^{\sharp}(\widetilde{S}_{\alpha}(\vec{f})^{2})\|_{L^{p/2}(w)}^{\frac{1}{2}}\\
 & \lesssim\alpha^{mn}(p+1)[w]_{A_{\infty}}^{\frac{1}{2}}\|\mathcal{M}(\vec{f})^{2}\|_{L^{p/2}(w)}^{\frac{1}{2}}\\
 & =\alpha^{mn}(p+1)[w]_{A_{\infty}}^{\frac{1}{2}}\|\mathcal{M}(\vec{f})\|_{L^{p}(w)}
\end{align*}
where we have used that for suitable choices of $\gamma$,
\[
M_{\gamma}^{\sharp}(\widetilde{S}_{\alpha,\psi}(\vec{f})^{2})(x)
\lesssim \alpha^{2mn} \M(\vec{f})(x)^{2}, \quad x \in \Rn.
\]

Hence to end the proof of Theorem \ref{thm:CF}, it remains to settle that pointwise estimate.
\begin{lemma}\label{lem:Msharp}
For every $\alpha \geq 1$ and $0<\gamma<\frac{1}{2m}$, we have
\begin{align}
M_{\gamma}^{\sharp}(\widetilde{S}_{\alpha,\psi}(\vec{f})^{2})(x)
\lesssim \alpha^{2mn} \M(\vec{f})(x)^{2}, \quad x \in \Rn.
\end{align}
\end{lemma}

\begin{proof}
Let $x \in Q$. It suffices to show that for some $c_{Q}$ chosen later
\begin{equation}\label{eq:JSM}
\mathcal{J}:=\bigg(\fint_{Q}|\widetilde{S}_{\alpha}(\vec{f})^{2}(x)-c_{Q}|^{\gamma}dx\bigg)^{\frac{1}{\gamma}}
\lesssim \alpha^{2mn}\mathcal{M}(\vec{f})(x)^{2}.
\end{equation}
For a cube $Q\subset\R^n$, we set $T(Q)=Q\times(0,\ell(Q))$. We then write
\[
\widetilde{S}_{\alpha,\psi}(\vec{f})^{2}(x) =E(\vec{f})(x)+F(\vec{f})(x),
\]
where
\begin{align*}
E(\vec{f})(x) &:=\iint_{T(2Q)}\Phi\Big(\frac{x-y}{\alpha t}\Big)|\psi_{t}(\vec{f})(y)|^{2}\frac{dydt}{t^{n+1}},
\\
F(\vec{f})(x) &:= \iint_{\mathbb{R}_{+}^{n+1}\backslash T(2Q)}\Phi\Big(\frac{x-y}{\alpha t}\Big)|\psi_{t}(\vec{f})(y)|^{2}\frac{dydt}{t^{n+1}}.
\end{align*}
Let us choose $c_{Q}=F(\vec{f})(x_{Q})$ where $x_{Q}$ is the center
of $Q$. Then we have that
\begin{align}\label{eq:JJJ}
\mathcal{J}& \lesssim\bigg(\fint_{Q}|E(\vec{f})(x)|^{\gamma}dx\bigg)^{\frac{1}{\gamma}}
+\bigg(\fint_{Q}|F(\vec{f})(x)-F(\vec{f})(x_{Q})|^{\gamma}dx\bigg)^{\frac{1}{\gamma}}
=:\mathcal{J}_1+ \mathcal{J}_2.
\end{align}
Let us first focus on $\mathcal{J}_1$. Set $\vec{f}^0:=(f^0_1,\dotsc,f^0_m)$, $f_{i}^{0}=f_{i}\chi_{Q^{*}}$, and $f_{i}^{\infty}=f_{i}\chi_{(Q^{*})^{c}}$, $i=1,\dots,m$, where $Q^{*}=8Q$. Then we have
\begin{equation}\label{eq:E0Ea}
E(\vec{f})(x)\lesssim E(\vec{f}^0)(x)+\sum_{\alpha\in\mathcal{I}_{0}}E(f_{1}^{\alpha_1},\ldots,f_{m}^{\alpha_m})(x),
\end{equation}
where $\mathcal{I}_{0}:=\{\alpha=(\alpha_{1},\dots,\alpha_{m}):\,\alpha_{i}\in\{0,\infty\},\ \text{ and at least one \ensuremath{\alpha_{i}\neq0}}\}$. Using Kolmogorov's inequality and \eqref{eq:SS-end}, we have
\begin{align}\label{eq:Ef0}
\bigg(\fint_{Q}|E(\vec{f}^0)(x)|^{\gamma} & dx\bigg)^{\frac{1}{\gamma}}
\leq\bigg(\fint_{Q}|\widetilde{S}_{\alpha,\psi}(\vec{f}^0)|^{2\gamma}dx\bigg)^{\frac{2}{2\gamma}}
\nonumber\\
& \lesssim\|\widetilde{S}_{\alpha,\psi}(\vec{f}^0)\|_{L^{1/m,\infty}(Q,\frac{dx}{|Q|})}^{2}
\lesssim \alpha^{2mn}\bigg(\prod_{j=1}^{m} \fint_{Q} |f_{j}| dx\bigg)^{2}.
\end{align}
On the other hand, for each $\alpha \in\mathcal{I}_{0}$,
\begin{align}\label{eq:EFa-1}
\bigg(\fint_{Q}|E(\vec{f}^{\alpha})(x)|^{\gamma}dx\bigg)^{\frac{1}{\gamma}}
&\lesssim\frac 1{|Q|}\int_{\Rn}\iint_{T(2Q)}\Phi\Big(\frac{x-y}{\alpha t}\Big)|\psi_{t}(\vec{f}^{\alpha})(y)|^{2}\frac{dydt}{t^{n+1}}dx
\nonumber\\
&\lesssim\frac{1}{|Q|}\iint_{T(2Q)} (\alpha t)^n |\psi_{t}(\vec{f}^{\alpha})(y)|^{2}\frac{dydt}{t^{n+1}},
\end{align}
since $\int_{\mathbb{R}^{n}}\Phi\big(\frac{x-y}{\alpha t}\big)dx\leq c_{n}(\alpha t)^{n}$. By size estimate, for $y\in2Q$ and $\alpha \in\mathcal{I}_{0}$, one has
\begin{align}\label{eq:EFa-2}
|\psi_{t}(\vec{f}^{\alpha})(y)| & \lesssim \bigg(\frac{t}{\ell(Q)}\bigg)^{\delta}
\sum_{k=0}^{\infty} 2^{-k\delta} \bigg(\prod_{j=1}^{m} \fint_{2^{k}Q}|f_j| dx \bigg).
\end{align}
Then, \eqref{eq:EFa-1} and \eqref{eq:EFa-2} give that for every $\alpha \in\mathcal{I}_{0}$,
\begin{align} \label{eq:Efa}
&\bigg(\fint_{Q} |E(\vec{f}^{\alpha})(x)|^{\gamma}dx\bigg)^{\frac{1}{\gamma}}
\nonumber\\
& \lesssim \bigg[\sum_{k=0}^{\infty} 2^{-k\delta} \bigg(\prod_{j=1}^{m}\fint_{2^k Q}|f_j|\bigg)\bigg]^{2}
\frac{\alpha^n}{|Q|} \iint_{T(2Q)} \bigg(\frac{t}{\ell(Q)}\bigg)^{2\delta}dy \frac{dt}{t}
\nonumber\\
& \lesssim \alpha^{n}\bigg[\sum_{k=0}^{\infty} 2^{-k\delta}\bigg(\prod_{j=1}^{m}\fint_{2^kQ}|f_j|\, dx\bigg)\bigg]^{2}
\nonumber\\
& \lesssim \alpha^{n}\sum_{k=0}^{\infty} 2^{-k\delta} \bigg(\prod_{j=1}^{m}\fint_{2^k Q}|f_j| \, dx\bigg)^{2},
\end{align}
where the Cauchy-Schwarz inequality was used in the last inequality. Gathering \eqref{eq:E0Ea}, \eqref{eq:Ef0} and \eqref{eq:Efa}, we obtain
\begin{align}\label{eq:J1}
\mathcal{J}_1 \lesssim \alpha^{2mn} \M(\vec{f})(x).
\end{align}
To complete the proof it remains to provide a bound for $\mathcal{J}_2$. From \cite[eq. (4.6)]{BH}, we have that for any $x\in Q$,
\begin{equation}\label{eq:J2}
|F(\vec{f})(x)-F(\vec{f})(x_{Q})|
\lesssim \alpha^{2mn}\sum_{k=0}^{\infty} 2^{-k\delta} \bigg(\prod_{j=1}^m \fint_{2^k Q} |f_j|\, dx \bigg)^2,
\end{equation}
Hence, \eqref{eq:JSM} is a consequence of \eqref{eq:JJJ}, \eqref{eq:J1} and \eqref{eq:J2}.
\end{proof}

\subsection{Proof of Theorem \ref{thm:mixed}}
In view of \eqref{eq:gSS} and $\lambda>2m$, it is enough to present the proof of \eqref{eq:Smixed}. We use a hybrid of the arguments in \cite{CMP} and \cite{LOPi}. Define
\begin{align*}
\mathcal{R}h(x)=\sum_{j=0}^{\infty} \frac{T^j_{u} h(x)}{2^j K_0^j},
\end{align*}
where $K_0>0$ will be chosen later and $T_{u}f(x) := M(f u)(x)/u(x)$ if $u(x) \neq 0$, $T_{u}f(x)=0$ otherwise. It immediately yields that
\begin{align}\label{eq:R-1}
h \leq \mathcal{R}h \quad
\text{and}\quad T_u(\mathcal{R}h) \leq 2K_0 \mathcal{R}h.
\end{align}
Moreover, we claim that for some $r>1$,
\begin{align}\label{eq:R-2}
\mathcal{R}h \cdot u v^{\frac{1}{mr'}} \in A_{\infty} \quad\text{and}\quad
\|\mathcal{R}h\|_{L^{r',1}(u v^{\frac1m})} \leq 2 \|h\|_{L^{r',1}(u v^{\frac1m})}.
\end{align}
The proofs will be given at the end of this section.

Note that
\begin{equation}\label{e:Lpq}
\|f^q\|_{L^{p,\infty}(w)}= \|f\|^q_{L^{pq,\infty}(w)}, \ \ 0<p,q<\infty.
\end{equation}
This implies that
\begin{align*}
& \bigg\| \frac{S_{\alpha}(\vec{f})}{v}
\bigg\|_{L^{\frac{1}{m},\infty}(u v^{\frac1m})}^{\frac{1}{mr}}
= \bigg\| \bigg(\frac{S_{\alpha}(\vec{f})}{v}\bigg)^{\frac{1}{mr}}
\bigg\|_{L^{r,\infty}(u v^{\frac1m})}
\\ 
&=\sup_{\|h\|_{L^{r',1}(u v^{\frac1m})}=1}
\bigg|\int_{\Rn} |S_{\alpha}(\vec{f})(x)|^{\frac{1}{mr}} h(x)
u(x) v(x)^{\frac{1}{mr'}} dx  \bigg|
\\ 
&\leq \sup_{\|h\|_{L^{r',1}(u v^{\frac1m})}=1}
\int_{\Rn} |S_{\alpha}(\vec{f})(x)|^{\frac{1}{mr}} \mathcal{R}h(x)
u(x) v(x)^{\frac{1}{mr'}} dx.
\end{align*}
Invoking Theorem \ref{thm:CF} and H\"{o}lder's inequality, we obtain
\begin{align*}
& \int_{\Rn} |S_{\alpha}(\vec{f})(x)|^{\frac{1}{mr}} \mathcal{R}h(x)
u(x) v(x)^{\frac{1}{mr'}} dx
\\ 
& \lesssim
\int_{\Rn} \M(\vec{f})(x)^{\frac{1}{mr}} \mathcal{R}h(x)
u(x) v(x)^{\frac{1}{mr'}} dx
\\ 
& =
\int_{\Rn} \bigg(\frac{\M(\vec{f})(x)}{v(x)}\bigg)^{\frac{1}{mr}}
\mathcal{R}h(x) u(x) v(x)^{\frac{1}{m}} dx
\\ 
& \leq \bigg\|\bigg(\frac{\M(\vec{f})}{v}\bigg)^{\frac{1}{mr}}
\bigg\|_{L^{r,\infty}(u v^{\frac1m})}
\|\mathcal{R}h\|_{L^{r',1}(u v^{\frac1m})}
\\ 
& \leq \bigg\|\frac{\M(\vec{f})}{v}\bigg\|_{L^{\frac1m,\infty}
(u v^{\frac1m})}^{\frac{1}{mr}}
\|h\|_{L^{r',1}(u v^{\frac1m})},
\end{align*}
where we used \eqref{e:Lpq} and \eqref{eq:R-2} in the last inequality. Here we need to apply the weighted mixed weak type estimates for $\M$ proved in Theorems 1.4 and 1.5 in \cite{LOPi}. Consequently, collecting the above estimates,
we get the desired result
\begin{equation*}
\bigg\| \frac{S_{\alpha}(\vec{f})}{v}
\bigg\|_{L^{\frac{1}{m},\infty}(u v^{\frac1m})}
\lesssim \bigg\| \frac{\M(\vec{f})}{v}
\bigg\|_{L^{\frac1m,\infty}(u v^{\frac1m})}
\lesssim \prod_{i=1}^m \| f_i \|_{L^1(w_i)}.
\end{equation*}

It remains to show our foregoing claim \eqref{eq:R-2}. The proof follows the same scheme of that in \cite{CMP}. For the sake of completeness we here give the details. Together with Lemma \ref{lem:weight}, the hypothesis $(1)$ or $(2)$ indicates that $u \in A_1$ and $v^{\frac1m} \in A_{\infty}$. The former implies that
\begin{equation}\label{e:S-1}
\|T_{u}f\|_{L^{\infty}(u v^{\frac1m})}
\leq [u]_{A_1} \|f\|_{L^{\infty}(u v^{\frac1m})}.
\end{equation}
The latter yields that $v^{\frac1m} \in A_{q_0}$ for some $q_0>1$. It follows from $A_p$ factorization theorem that there exist
$v_1,v_2 \in A_1$ such that $v^{\frac1m}=v_1 v_2^{1-q_0}$.

Additionally, it follows from Lemma 2.3 in \cite{CMP} that if $u_1,u_2 \in A_1$, then there exists $\epsilon_0=\epsilon_0([u_1]_{A_1},[u_2]_{A_1}) \in (0,1)$ such that $u_1 v_1^{\epsilon} \in A_{p_1}$ and $u_2 v_2^{\epsilon} \in A_{p_2}$ for any $0 < \epsilon < \epsilon_0$, $v_1 \in A_{p_1}$ and $v_2 \in A_{p_2}$, $1 \leq p_1,p_2 < \infty$. Then $u v_2^{\frac{q_0-1}{p_0-1}} \in A_1$ if we set $p_0>1+(q_0-1)/{\epsilon_0}$. Thus, we have
\begin{align*}
u^{1-p_0} v^{\frac1m}
=v_1 \Big(u v_2^{\frac{q_0-1}{p_0-1}}\Big)^{1-p_0} \in A_{p_0}.
\end{align*}
It immediately implies that
\begin{align}\label{e:S-2}
\|T_{u}f\|_{L^{p_0}(u v^{\frac1m})}
= \|M(fu)\|_{L^{p_0}(u^{1-p_0} v^{\frac1m})}
\leq c_1 \|f\|_{L^{p_0}(u v^{\frac1m})}.
\end{align}
By \eqref{e:S-1}, \eqref{e:S-2} and Marcinkiewicz interpolation in \cite[~Proposition A.1]{CMP}, we have $T_u$ is bounded on
$L^{p,1}(u v^{\frac1m})$ for all $p \in (p_0,\infty)$ with the constant
\begin{align*}
K(p)=2^{\frac1p} \bigg(c_1 \Big(\frac{1}{p_0}-\frac{1}{p}\Big)^{-1} + c_2\bigg),
\end{align*}
and $c_2:=[v]_{A_1}$. Note that $K(p)$ is decreasing with respect to $p$. Hence, we obtain
\begin{equation}\label{e:Lp1}
\|T_{u}f\|_{L^{p,1}(u v^{\frac1m})} \leq K_0 \|f\|_{L^{p,1}(u v^{\frac1m})}, \ \forall p \geq 2p_0
\end{equation}
where $K_0 := 4p_0(c_1+c_2) > K(2p_0) \geq K(p)$.

The inequality \eqref{eq:R-1} indicates that $\mathcal{R}h \cdot u \in A_1$ with $[\mathcal{R}h \cdot u]_{A_1} \leq 2K_0$.
Let $0<\epsilon<\min\{\epsilon_0,\frac{1}{2p_0}\}$, and $r=(\frac{1}{\epsilon})'$. Then $(\mathcal{R}h \cdot u) v_1^{\epsilon} \in A_1$, and the second inequality in \eqref{eq:R-2} follows from \eqref{e:Lp1}. By $A_p$ factorization theorem again, we obtain
\begin{align*}
\mathcal{R}h \cdot u v^{\frac{1}{mr'}}
=[(\mathcal{R}h \cdot u) v_1^{\epsilon}] \cdot v_2^{1-[(q_0-1)\epsilon+1]}
\in A_{(q_0-1)\epsilon+1} \subset A_{\infty}.
\end{align*}
The proof is complete.
\qed

\section{Local decay estimates}

To show Theorem \ref{thm:local}, we need the following Carleson embedding theorem from \cite[Theorem~4.5]{HP}.
\begin{lemma}\label{lem:Car-emb}
Suppose that the sequence $\{a_Q\}_{Q \in \D}$ of nonnegative numbers satisfies the Carleson packing condition
\begin{align*}
\sum_{Q \in \D:Q \subset Q_0} a_Q \leq A w(Q_0),\quad\forall Q_0 \in \D.
\end{align*}
Then for all $p \in (1, \infty)$ and $f \in L^p(w)$,
\begin{align*}
\bigg(\sum_{Q \in \D} a_Q \Big(\frac{1}{w(Q)} \int_{Q} f(x) w \ dx\Big)^p \bigg)^{\frac1p}
\leq A^{\frac1p} p' \|f\|_{L^p(w)}.
\end{align*}
\end{lemma}

We also need a local version of Coifman-Fefferman inequality with the precise $A_p$ norm.
\begin{lemma}
For every $1<p<\infty$ and $w \in A_p$, we have
\begin{align}
\label{eq:CF-local} \|S_{\alpha}(\vec{f})\|_{L^2(Q, w)}
&\leq c_{n,p} \alpha^{mn}[w]_{A_p}^{\frac12} \|\M(\vec{f})\|_{L^2(Q, w)},
\\
\label{eq:CF-local-g} \|g_{\lambda}^*(\vec{f})\|_{L^2(Q, w)}
&\leq \frac{c_{n,p}}{1-2^{-n(\lambda-2m)/2}} [w]_{A_p}^{\frac12} \|\M(\vec{f})\|_{L^2(Q, w)},
\end{align}
for every cube $Q$ and $f_j\in L_c^\infty$ with $\supp f_j\subset Q$ $(j=1,\dots,m)$.
\end{lemma}

\begin{proof}
Let $w \in A_p$ with $1<p<\infty$. Fix a cube $Q \subset \Rn$. Recall the definition of $\widetilde{S}_{\alpha}$ in \eqref{eq:S-def}. Pick $0<\epsilon<\frac{1}{2m}$. By \eqref{e:mfQ}, Kolmogorov's inequality, \eqref{eq:SS-end} and
$f_j\in L_c^\infty$ with $\supp f_j\subset Q$, $j=1,\dots,m$, we have
\begin{align*}
m_{\widetilde{S}_{\alpha}(\vec{f})^2}(Q)
& \lesssim \|\widetilde{S}_{\alpha}(\vec{f})^2\|_{L^{\epsilon}(Q,\frac{dx}{|Q|})}
\lesssim \|\widetilde{S}_{\alpha}(\vec{f})\|_{L^{1/m,\infty}(Q,\frac{dx}{|Q|})}^2
\\ 
& \lesssim \alpha^{2mn}\bigg(\prod_{i=1}^m \fint_{Q} |f_i| dx\bigg)^2
\leq \alpha^{2mn}\inf_{x \in Q} \M(\vec{f})(x)^2,
\end{align*}
which implies that
\begin{align}\label{eq:mSf}
m_{\widetilde{S}_{\alpha}(\vec{f})^2}(Q) w(Q)
\lesssim \alpha^{2mn}\int_{Q} \M(\vec{f})(x)^2 w(x) dx.
\end{align}
On the other hand, from \cite[Proposition~4.1]{BH},
one has for every cube $Q'$,
\begin{align}\label{eq:SM}
\omega_{\lambda}(\widetilde{S}_{\alpha}(\vec{f})^2; Q')
&\lesssim \alpha^{2mn} \sum_{j=0}^{\infty} 2^{-j\delta_0}
\bigg(\prod_{i=1}^m \fint_{2^j Q'} |f_i(y_i)| dy_i\bigg)^2
\nonumber\\
&\lesssim \alpha^{2mn}\sum_{j=0}^{\infty} 2^{-j\delta_0}
\inf_{Q'} \M(\vec{f})^2
\lesssim \alpha^{2mn}\inf_{Q'} \M(\vec{f})^2,
\end{align}
where $0<\delta_0<\min\{\delta, \frac12\}$.
Thus, together with  \eqref{eq:mSf} and \eqref{eq:SM},
the estimate \eqref{eq:mf} applied to $Q_0=Q$ and
$f=\widetilde{S}_{\alpha}(\vec{f})^2$ gives that
\begin{align*}
\|\widetilde{S}_{\alpha}(\vec{f})\|_{L^2(Q, w)}^2
&\lesssim m_{\widetilde{S}_{\alpha}(\vec{f})^2}(Q) w(Q)
+ \sum_{Q' \in \S(Q)} \omega_{2^{-n-2}}(\widetilde{S}_{\alpha}(\vec{f})^2; Q') w(Q')
\\
&\lesssim \alpha^{2mn}\|\M(\vec{f})\|_{L^2(Q, w)}^2
+ \alpha^{2mn}\sum_{Q' \in \S(Q)} \inf_{Q'} \M(\vec{f})^2 w(Q').
\end{align*}
From this and \eqref{eq:SSS}, we see that to obtain \eqref{eq:CF-local}, it suffices to prove
\begin{align}\label{eq:QSQ}
\sum_{Q' \in \S(Q)} \inf_{Q'} \M(\vec{f})^2 w(Q') \lesssim [w]_{A_p} \|\M(\vec{f})\|_{L^2(Q, w)}^2.
\end{align}

Recall that a new version of $A_{\infty}$ was introduced by Hyt\"{o}nen and P\'{e}rez \cite{HP}:
\begin{align*}
[w]'_{A_{\infty}} := \sup_{Q} \frac{1}{w(Q)} \int_{Q} M(w \mathbf{1}_Q)(x) dx.
\end{align*}
By \cite[Proposition~2.2]{HP}, there holds
\begin{align}\label{eq:AiAi}
c_n [w]'_{A_{\infty}} \le [w]_{A_{\infty}} \leq [w]_{A_p}.
\end{align}
Observe that for every $Q'' \in \D$,
\begin{align*}
\sum_{Q' \in \S(Q): Q' \subset Q''} w(Q')
&=\sum_{Q' \in \S(Q): Q' \subset Q''} \langle w \rangle_{Q'} |Q'|
\lesssim \sum_{Q' \in \S(Q): Q' \subset Q''} \inf_{Q'} M(w \mathbf{1}_{Q''}) |E_{Q'}|
\\
&\lesssim \int_{Q''} M(w \mathbf{1}_{Q''})(x) dx
\leq [w]'_{A_{\infty}} w(Q'') \lesssim [w]_{A_p} w(Q''),
\end{align*}
where we used the disjointness of $\{E_{Q'}\}_{Q' \in \S(Q)}$ and \eqref{eq:AiAi}. This shows that the collection $\{w(Q')\}_{Q' \in \S(Q)}$ satisfies the Carleson packing condition with the constant $c_n [w]_{A_p}$. As a consequence, this and Lemma \ref{lem:Car-emb} give that
\begin{align*}
\sum_{Q' \in \S(Q)} \inf_{Q'} \M(\vec{f})^2 w(Q')
&\le \sum_{Q' \in \S(Q)} \bigg(\frac{1}{w(Q')} \int_{Q'} \M(\vec{f})\, \mathbf{1}_Q w\, dx \bigg)^2 w(Q')
\\
&\lesssim [w]_{A_p} \|\M(\vec{f}) \mathbf{1}_Q\|_{L^2(w)}^2
=[w]_{A_p} \|\M(\vec{f}) \mathbf{1}_Q\|_{L^2(Q, w)}^2,
\end{align*}
where the above implicit constants are independent of $[w]_{A_p}$ and $Q$. This shows \eqref{eq:QSQ} and completes the proof of \eqref{eq:CF-local}.

Finally,  the estimate \eqref{eq:CF-local-g} immediately follows from \eqref{eq:CF-local} and the fact that
\begin{equation*}
 g^*_{\lambda}(\vec{f})(x) \le S_{1}(\vec{f})(x)+\sum_{k=0}^{\infty} 2^{-\frac{k\lambda n} {2}}S_{2^{k+1}}(\vec{f})(x).
\end{equation*}
This completes the proof.
\end{proof}

\vspace{0.2cm}
\begin{proof}[\textbf{Proof of Theorem \ref{thm:local}.}]
Let $p>1$ and $r>1$ be chosen later. Define the Rubio de Francia algorithm:
\begin{align*}
\mathcal{R}h=\sum_{k=0}^{\infty} \frac{M^{k}h}{2^k\|M\|^k_{L^{r'}\to L^{r'}}}.
\end{align*}
Then it is obvious that
\begin{align}\label{eq:hRh}
h \le \mathcal{R}h \quad\text{and}\quad \|\mathcal{R}h\|_{L^{r'} (\Rn)} \leq 2 \|h\|_{L^{r'} (\Rn)}.
\end{align}
Moreover, for any nonnegative $h \in L^{r'}(\Rn)$, we have that $\mathcal{R}h \in A_1$ with
\begin{align}\label{eq:Rh-A1}
[\mathcal{R}h]_{A_1} \leq 2 \|M\|_{L^{r'} \to L^{r'}} \leq c_n \ r.
\end{align}

By Riesz theorem and the first inequality in \eqref{eq:hRh}, there exists some nonnegative function $h \in L^{r'}(Q)$ with $\|h\|_{L^{r'}(Q)}=1$ such that
\begin{align}\label{eq:FQ}
\mathscr{F}_Q^{\frac1r} &:= |\{x \in Q:  S_{\alpha}(\vec{f})(x) > t \M(\vec{f})(x)\}|^{\frac1r}
\nonumber \\ 
&= |\{x \in Q:  S_{\alpha}(\vec{f})(x)^2 > t^2 \M(\vec{f})(x)^2\}|^{\frac1r}
\nonumber \\ 
& \leq \frac{1}{t^2} \bigg\| \bigg(\frac{S_{\alpha}(\vec{f})}{\M(\vec{f})}\bigg)^2 \bigg\|_{L^r(Q)}
\leq \frac{1}{t^2} \int_{Q} S_{\alpha}(\vec{f})^2 \, h \, \M(\vec{f})^{-2} dx
\nonumber \\ 
&\leq t^{-2} \|S_{\alpha}(\vec{f})\|_{L^2(Q, w)}^2,
\end{align}
where $w=w_1 w_2^{1-p}$, $w_1= \mathcal{R}h$ and $w_2 = \M(\vec{f})^{2(p'-1)}$. Recall that the m-linear version of Coifmann-Rochberg theorem  \cite[Lemma~1]{OPR} asserts that
\begin{align}\label{eq:C-R}
[(\M(\vec{f}))^{\delta}]_{A_1} \leq \frac{c_n}{1-m\delta}, \quad\forall \delta \in (0, \frac1m).
\end{align}
In view of \eqref{eq:Rh-A1} and \eqref{eq:C-R}, we see that $w_1, w_2 \in A_1$ provided $p>2m+1$. Then the reverse $A_p$ factorization theorem gives that
$w=w_1 w_2^{1-p} \in A_p$ with
\begin{align}\label{eq:w-Ap-r}
[w]_{A_p} \leq [w_1]_{A_1} [w_2]_{A_1}^{p-1} \leq c_n ~ r.
\end{align}
Thus, gathering \eqref{eq:CF-local}, \eqref{eq:FQ} and \eqref{eq:w-Ap-r},
we obtain
\begin{align*}
\mathscr{F}_Q^{\frac1r}
& \le c_{n} t^{-2}\alpha^{2mn} [w]_{A_p} \|\M(\vec{f})\|_{L^2(Q, w)}^2
 \\ 
& = c_{n} t^{-2} \alpha^{2mn}[w]_{A_p} \|\mathcal{R}h\|_{L^1(Q)}
 \\ 
&\le c_{n} t^{-2} \alpha^{2mn}[w]_{A_p} \|\mathcal{R}h\|_{L^{r'}(Q)}  |Q|^{\frac1r}
 \\ 
&\le c_{n} t^{-2}\alpha^{2mn} [w]_{A_p} \|h\|_{L^{r'}(Q)}  |Q|^{\frac1r}
 \\ 
&\le c_{n} r t^{-2}\alpha^{2mn} |Q|^{\frac1r}.
\end{align*}

Consequently, if $t> \sqrt{c_n e}\,\alpha^{mn}$, choosing $r>1$
so that $t^2/e = c_n \alpha^{2mn}r$, we have
\begin{align}\label{eq:FQr-1}
\mathscr{F}_Q \le (c_n \alpha^{2mn}r t^{-2})^r |Q| = e^{-r} |Q|
= e^{-\frac{t^2}{c_n e\alpha^{2mn}}} |Q|.
\end{align}
If $0<t \le \sqrt{c_n e}\alpha^{mn}$, it is easy to see that
\begin{equation}\label{eq:FQr-2}
\mathscr{F}_Q \le |Q| \le e \cdot e^{-\frac{t^2}{c_n e\alpha^{2mn}}} |Q|.
\end{equation}
Summing \eqref{eq:FQr-1} and \eqref{eq:FQr-2} up, we deduce that
\begin{equation*}
\mathscr{F}_Q=|\{x \in Q:  S_{\alpha}(\vec{f})(x) > t \M(\vec{f})(x)\}|
\le c_1 e^{-c_2 t^2/\alpha^{2mn}} |Q|,\quad\forall t>0.
\end{equation*}
This proves \eqref{eq:local-1}.

To obtain \eqref{eq:local-2}, we use the same strategy and \eqref{eq:CF-local-g} in place of \eqref{eq:CF-local}.
\end{proof}

Next we present another proof of Theorem \ref{thm:local}. In view of \eqref{eq:S-sparse} and \eqref{eq:g-sparse}, following the approach in \cite{PRR}, it suffices to prove the following.
\begin{lemma}
There exist $c_1>0$ and $c_2>0$ such that for every sparse family $\S \subset \D$ and for every cube $Q_0$,
\[
|\{x\in Q_0: \A_{\S}^{2}(\vec{f})>t \mathcal{M}(\vec{f})\}| \leq c_1 e^{-c_2 t^2} |Q_0|.
\]
where $\vec{f}=(f_1,\ldots,f_m)$ are supported on $Q_0$.
\end{lemma}
\begin{proof}
Fix a sparse family $\S \subset \D$ and a cube $Q_0$. First we observe that
\begin{align*}
\mathcal{K}:=|\{x\in Q_0: \A_{\S}^{2}(\vec{f})>t \M(\vec{f})\}|
&=\Big|\Big\{x\in Q_0: \sum_{Q\in\S} \prod_{i=1}^{m}\langle |f_i|\rangle_Q^2 > t^2 \M(\vec{f})^2 \Big\}\Big|.
\end{align*}
Now we consider the family of at most $3^n$ cubes $Q_{j}\in\D$ such that $|Q_{j}|\simeq|Q_{0}|$ and $|Q_{j}\cap Q_{0}|>0$. We have that adding those cubes to $\text{\ensuremath{\mathcal{S}}}$ it remains a sparse family, we shall assume then that $Q_{j}\in\S$. For such $Q_j$, we define
\begin{align*}
T_j^1(\vec{f}):=\sum_{Q\in\S: Q \subset Q_j} \prod_{i=1}^{m}\langle |f_i|\rangle_Q^2 \mathbf{1}_Q \quad\text{and}\quad
T_j^2(\vec{f}):=\sum_{Q\in\S: Q \supsetneq Q_j} \prod_{i=1}^{m}\langle |f_i|\rangle_Q^2 \mathbf{1}_{Q}.
\end{align*}
Then, one has
\begin{align*}
\mathcal{K} &\le \sum_{j=1}^{3^{n}}
\Big|\Big\{x\in Q_j: T_j^1(\vec{f}) + T_j^2(\vec{f}) > t^2 \mathcal{M}(\vec{f})^2 \Big\}\Big|
\\
& \leq\sum_{j=1}^{3^n} \sum_{i=1}^2 \left|\left\{x\in Q_j\,: T_j^i(\vec{f})>c_n t^2 \M(\vec{f})^2\right\}\right|
=:\sum_{j=1}^{3^{n}} (\mathcal{K}_j^1 + \mathcal{K}_j^2).
\end{align*}

We recall that in \cite[Theorem 2.1]{OCPR}, it was established that
\begin{equation}\label{eq:ExpDecay}
\bigg|\bigg\{ x\in Q\,:\sum_{Q'\in\S,\,Q'\subseteq Q}\mathbf{1}_{Q'}(x)>t\bigg\} \bigg|\leq ce^{-\alpha t}|Q|, \quad\forall Q.
\end{equation}
For $\mathcal{K}^1_j$, taking into account \eqref{eq:ExpDecay}, we obtain
\begin{align*}
\mathcal{K}^1_j & \leq\bigg|\bigg\{x\in Q_j\,:\sum_{Q\in\S, Q\subset Q_j} {\bf 1}_{Q}(x)>c_{n}t^2 \bigg\}\bigg|
\leq c e^{-\alpha t^2}|Q_{j}|\simeq c e^{-\alpha t^2}|Q_{0}|.
\end{align*}
For $\mathcal{K}^1_j$, since $\vec{f}$ is supported in $Q_0$, we deduce that
\begin{align*}
\mathcal{K}_j^2 &\leq\Big|\Big\{x\in Q_{j}\,: T_j^2(\vec{f} \cdot {\bf 1}_{Q_0}) > c_n t^2 \prod_{i=1}^m \langle |f_i| \rangle_{Q_0}^2\Big\} \Big|
\\
& \leq \Big|\Big\{x \in Q_j\,: \sum_{Q\in\S, Q\supsetneq Q_j} \Big(\prod_{i=1}^{m} |Q_0|/|Q|\Big)^{2} {\bf 1}_{Q}(x)>c_{n}t^{2}\Big\} \Big|
\\
& \leq \Big|\Big\{ x\in Q_j\,:\sum_{j=1}^{\infty} 2^{-2mj}>c_n t^2 \Big\} \Big|.
\end{align*}
Observe that if $t$ is large enough, then
\[
\Big|\Big\{ x\in Q_{j}\,:\sum_{j=1}^{\infty} 2^{-2mj}>c_{n}t^{2}\Big\}\Big|=0.
\]
Consequently,
\[
\mathcal{K}_j^2 \lesssim e^{-t^2} |Q_0|.
\]
We are done.
\end{proof}



\begin{thebibliography}{999}


\bibitem{ACM}T. Anderson, D. Cruz-Uribe and K. Moen,
\emph{Logarithmic bump conditions for Calder\'{o}n-Zygmund operators on spaces of homogeneous type},
Publ. Mat. \textbf{59} (2015), 17--43.


\bibitem{BH}T. A. Bui and M. Hormozi,
\emph{Weighted bounds for multilinear square functions},
Potential Anal. \textbf{46} (2017), 135--148.


\bibitem{B2} F. Berra,
\emph{From $A_1$ to $A_\infty$: new mixed inequalities for certain maximal operators},
arXiv:2006.03612


\bibitem{BCP}F. Berra, M. Carena and G. Pradolini,
\emph{Mixed weak estimates of Sawyer type for commutators of generalized singular integrals and related operators},
Mich. Math. J.\textbf{ 68 }(2019), 527--564.



\bibitem{BCP2}F. Berra, M. Carena and G. Pradolini,
\emph{Mixed weak estimates of Sawyer type for fractional integrals and some related operators},
J. Math. Anal. Appl.\textbf{ 479 }(2019), 1490--1505.

\bibitem{CRR} M. Caldarelli and I. P. Rivera-R\'{\i}os,
\emph{A sparse approach to mixed weak type inequalities.}
Math. Z. \textbf{296 }(2020), 787--812.


\bibitem{CXY}M. Cao, Q. Xue and K. Yabuta,
\emph{Weak and strong type estimates for the multilinear pseudo-differential operators},
J. Funct. Anal. \textbf{278 } (2020), 108454.


\bibitem{CY}M. Cao and K. Yabuta,
\emph{The multilinear Littlewood-Paley operators with minimal regularity conditions},
J. Fourier Anal. Appl. \textbf{25} (2019), 1203--1247.


\bibitem{COV}C. Cascante, J.M. Ortega and I.E. Verbitsky,
\emph{Nonlinear potentials and two weight trace inequalities for general dyadic and radial kernels},
Indiana Univ. Math. J. \textbf{53} (2004), 845--882.


\bibitem{CWX}S. Chen, H. Wu and Q. Xue,
\emph{A note on multilinear Muckenhoupt classes for multiple weights},
Studia Math. \textbf{223} (2014), 1--18.


\bibitem{CD}W. Chen and W. Dami\'an,
\emph{Weighted estimates for the multisublinear maximal function},
Rend. Circ. Mat. Palermo \textbf{62} (2013), 379--391.


\bibitem{CXY}X. Chen, Q. Xue and K. Yabuta,
\emph{On multilinear Littlewood-Paley operators},
Nonlinear Anal. \textbf{115} (2015), 25--40.


\bibitem{CDM}R. R. Coifman, D. Deng and Y. Meyer,
\emph{Domains de la racine carr\'ee de certains op\'erateurs diff\'erentiels accr\'etifs},
Ann. Inst. Fourier (Grenoble) \textbf{33} (1983), 123--134.


\bibitem{CMM}R. R. Coifman, A. McIntosh and Y. Meyer,
\emph{L'integrale de Cauchy definit un operateur borne sur $L^2$ pour les courbes lipschitziennes},
Ann. Math. \textbf{116} (1982), 361--387.


\bibitem{CM} R. Coifman and Y. Meyer,
\emph{On commutators of singular integral and bilinear singular integrals},
Trans. Amer. Math. Soc. \textbf{212} (1975), 315--331.


\bibitem{CMP}D. Cruz-Uribe, J. M. Martell and C. P\'{e}rez,
\emph{Weighted weak-type inequalities and a conjecture of Sawyer},
Int. Math. Res. Not. \textbf{30} (2005), 1849--1871.


\bibitem{CMP11}D. Cruz-Uribe, J. M. Martell and C. P\'{e}rez,
\emph{Weights, extrapolation and the theory of Rubio de Francia},
Operator Theory: Advances and Applications, Vol. 215,
Birkh\"auser/Springer Basel AG, Basel, 2011.


\bibitem{CP99}D. Cruz-Uribe and C. P\'erez,
\emph{Sharp two-weight, weak-type norm inequalities for singular integral operators},
Math. Res. Lett. \textbf{6} (1999), 1--11.


\bibitem{F}C. Fefferman,
\emph{Inequalities for strongly singular convolution operators},
Acta Math. \textbf{124} (1970), 9--36.


\bibitem{F2}C. Fefferman,
\emph{The uncertainty principle},
Bull. Amer. Math. Soc. \textbf{9}(1983), 129--206.


\bibitem{FJK1} E. B. Fabes, D. Jerison and C. Kenig,
\emph{Multilinear Littlewood--Paley estimates with applications to partial differential equations},
Proc. Natl. Acad. Sci. \textbf{79} (1982), 5746--5750.


\bibitem{FJK2} E. B. Fabes, D. Jerison and C. Kenig,
\emph{Necessary and sufficient conditions for absolute continuity of elliptic harmonic measure},
Ann. of Math. \textbf{119} (1984), 121--141.


\bibitem{FJK3} E. B. Fabes, D. Jerison and C. Kenig,
\emph{Multilinear square functions and partial differential equations},
Amer. J. Math. \textbf{107} (1985), 1325--1368.


\bibitem{Hy1}T. Hyt\"{o}nen,
\emph{The sharp weighted bound for general Calder\'{o}n-Zygmund operators},
Ann. of Math. (2) \textbf{175} (2012), 1473--1506.


\bibitem{Hy2}T. Hyt\"{o}nen,
\emph{The $A_2$ theorem: remarks and complements},
Contemp. Math. 612, Amer. Math. Soc., 91--106, Providence (2014).


\bibitem{HP}T. Hyt\"{o}nen and C. P\'{e}rez,
\emph{Sharp weighted bounds involving $A_{\infty}$},
Anal. PDE \textbf{6} (2013), 777--818.


\bibitem{Lacey}M.T. Lacey,
\emph{An elementary proof of the $A_2$ bound},
Israel J. Math. \textbf{217} (2017), 181--195.


\bibitem{LS}M.T. Lacey and S. Spencer,
\emph{On entropy bumps for Calder\'on-Zygmund operators},
Concr. Oper. \textbf{2} (2015), 47--52.


\bibitem{Ler11} A.K. Lerner,
\emph{Sharp weighted norm inequalities for Littlewood--Paley operators and singular integrals},
Adv. Math. \textbf{226} (2011), 3912--3926.


\bibitem{Ler13} A. K. Lerner,
\emph{A simple proof of the $A_2$ conjecture},
Int. Math. Res. Not. \textbf{14} (2013), 3159--3170.


\bibitem{Ler14}A. K. Lerner,
\emph{On sharp aperture-weighted estimates for square functions},
J. Fourier Anal. Appl. \textbf{20} (2014), 784--800.


\bibitem{LOPTT}A. K. Lerner, S. Ombrosi, C. P\'{e}rez, R. H. Torres and R. Trujillo-Gonz\'{a}lez,
\emph{New maximal functions and multiple weights for the multilinear Calder\'{o}n-Zygmund theory},
Adv. Math. \textbf{220} (2009), 1222--1264.


\bibitem{LMO}K. Li, J. M. Martell and S. Ombrosi,
\emph{Extrapolation for multilinear Muckenhoupt classes and applications},
Adv. Math. \textbf{373} (2020), 107286.


\bibitem{LOPi}K. Li, S. Ombrosi and B. Picardi,
\emph{Weighted mixed weak-type inequalities for multilinear operators},
Studia Math. \textbf{244} (2019), 203--215.


\bibitem{LOP}K. Li, S. Ombrosi and C. P\'{e}rez,
\emph{Proof of an extension of E. Sawyer's conjecture about weighted mixed weak-type estimates},
Math. Ann. \textbf{374 }(2019), 907--929.


\bibitem{Lich}L. Lichtenstein,
\emph{\"{U}ber die erste Randwertaufgabe der Potentialtheorie}
Sitzungsber. Berlin Math. Gesell., \textbf{15}, 92--96.


\bibitem{MW} B. Muckenhoupt and R. Wheeden,
\emph{ Some weighted weak-type inequalities for the Hardy-Littlewood maximal function and the Hilbert transform},
Indiana Math. J. \textbf{26} (1977), 801--816.


\bibitem{OmPR} S. Ombrosi, C. Pérez and J. Recchi
\emph{Quantitative weighted mixed weak-type inequalities for classical operators}
Indiana Univ. Math. J. \textbf{65} (2016), 615–640.

\bibitem{OPR}C. Ortiz-Caraballo, C. P\'{e}rez and E. Rela,
\emph{Exponential decay estimates for singular integral operators},
Math. Ann. \textbf{357} (2013), 1217--1243.


\bibitem{OCPR}C. Ortiz-Caraballo, C. P\'erez and E. Rela,
\emph{Improving bounds for singular operators via sharp reverse H\"older inequality for $A_{\infty}$},
Advances in harmonic analysis and operator theory, 303--321, Oper. Theory Adv. Appl., 229, Birkh\"auser/Springer Basel
AG, Basel, 2013.

\bibitem{PRR}C. P\'{e}rez and I. P. Rivera-R\'{\i}os, Three observations on commutators of singular integral operators with BMO functions. Harmonic analysis, partial differential equations, Banach spaces, and operator theory. Vol. 2, 287–304, Assoc. Women Math. Ser., 5, Springer, Cham, 2017.

\bibitem{PW}C. P\'{e}rez and R. Wheeden,
\emph{Uncertainty principle estimates for vector fields},
J. Funct. Anal. \textbf{181}(2001), 146--188.



\bibitem{S83}E.T. Sawyer,
\emph{Norm inequalities relating singular integrals and maximal function},
Studia Math. \textbf{75} (1983), 254--263.




\bibitem{SXY}S. Shi, Q. Xue and K. Yabuta,
\emph{On the boundedness of multilinear Littlewood--Paley $g^{*}_{\lambda}$ function},
J. Math. Pures Appl. \textbf{101} (2014), 394--413.


\bibitem{Wil}J. M. Wilson,
\emph{The intrinsic square function},
Rev. Mat. Iberoam. \textbf{23} (2007),  771--791.


\bibitem{XY}Q. Xue and J. Yan,
\emph{On multilinear square function and its applications to multilinear Littlewood-Paley operators with non-convolution type kernels},
J. Math. Anal. Appl. \textbf{422} (2015), 1342--1362.


\bibitem{ZK}P. Zorin-Kranich,
\emph{$A_p$-$A_{\infty}$ estimates for multilinear maximal and sparse operators},
J. Anal. Math. \textbf{138} (2019), 871--889.


\end{thebibliography}
\end{document}